\documentclass[10pt]{article}
\usepackage{amsfonts}
\usepackage{graphicx}
\usepackage{amssymb}
\usepackage{amsmath}
\usepackage{theorem}

\theoremheaderfont{\scshape}
\newtheorem{theorem}{Theorem}[section]
\newtheorem{corollary}[theorem]{Corollary}
\newtheorem{lemma}[theorem]{Lemma}

\theorembodyfont{\normalfont}
\newtheorem{remark}{Remark}[section]
\newenvironment{proof}{{\bf Proof }}{\hbox{~} \hfill \rule{0.5em}{0.5em}\\}
\numberwithin{equation}{section}

\newcommand{\ds}{\displaystyle}

\newcommand{\R}{\mathbb{R}}
\newcommand{\torus}{{\mathbb{T}}}
\newcommand{\gu}{\mathbf{u}}
\newcommand{\gm}{\mathbf{m}}
\newcommand{\zep}{{\mathcal{Z}^{\epsilon}}}

\title{Singularly perturbed degenerated parabolic equations and application to seabed 
 morphodynamics in tided environment}
\author{Ibrahima FAYE\thanks{Université de Bambey, BP 30 Bambey (Senegal), 
Ecole Doctorale de Mathématiques et Informatique. 
Laboratoire de Mathématiques de la décision et d'Analyse Numérique (L.M.D.A.N) F.A.S.E.G)/F.S.T. 
grandmbodj@hotmail.com}
\and
Emmanuel FRENOD\thanks{Universit\'e Europ\'eenne de Bretagne, Lab-STICC (UMR CNRS 3192),
Universit\'e de Bretagne-Sud, Centre Yves Coppens, Campus de Tohannic,
F-56017, Vannes, France.
emmanuel.frenod@univ-ubs.fr}
\and
Diaraf SECK\thanks{Universit\'{e} Cheikh Anta Diop de Dakar,
 BP 16 889, Dakar-Fann (Senegal).
Ecole Doctorale de Math\'ematiques et Informatique.
 Laboratoire de Math\'ematiques de la D\'ecision et
 d'Analyse Num\'erique (L.M.D.A.N) F.A.S.E.G/F.S.T
 {\bf \&}
 UMMISCO,  UMI 209, IRD, France.
 dseck@ucad.sn}}
\begin{document}
\maketitle

\pagestyle{myheadings}
 \renewcommand{\sectionmark}[1]{\markboth{#1}{}}
\renewcommand{\sectionmark}[1]{\markright{\thesection\ #1}}

\begin{abstract}\noindent
In this paper we build models for short-term, mean-term 
and long-term dynamics of dune and megariple morphodynamics. They are models that are
degenerated parabolic equations which are, moreover, singularly perturbed. We, then give an
existence and uniqueness result for the short-term and mean-term
models. This result is based on a time-space periodic solution existence result for degenerated parabolic equation that we set out. Finally the short-term model is homogenized. \\

\begin{center}{\bf Keywords} \end{center}Ò
Degenerated Parabolic Equation; Space-Time Periodic Solutions; Homogenization; 
Asymptotic Analysis; Asymptotic Expansion; Long Time Behavior; 
Dune and Megaripple Morphodynamics; Modeling Coastal Zone Phenomena. \\

\begin{center}{\bf AMS Mathematics  Subject Classification} \end{center}
35K65 (Degenerate parabolic equations); 
35B25  (Singular perturbations);
35B40  (Asymptotic behavior of solutions);
35B10  (Periodic solutions);
92F05   (Other natural sciences [Sedimentology]);
86A60   (Geological problems).
\end{abstract}
\section{Introduction and results}
Dune and megaripple generation and dynamics, on the seabed over a continental shelf, 
are the results of interaction between the seabed and water currents.
The study of the physical processes allowing for the generation
of dunes, or governing their evolution or stability involves modeling and numerical simulation. Roughly speaking, the models in use essentially couple an equation for the fluid fields (Navier-Stokes or shallow water equations) to an equation describing sand transport on the seabed. Those methods were used with success in
DeVriend \cite{DeVriendThese},
Engelund and Hansen \cite{EngHan},
Kennedy \cite{Kenn},
Blondeau \cite{Blondeau},
Dawson,  Johns  and Soulsby   \cite{DawJohSoul},
Johns, Soulsby and Chesher \cite{JohSouChe},
 Idier \cite{Idier}
and Idier, Astruc and Hulsher \cite{IdierAsH}.\\
A careful watch reveals that the use of numerical simulation for the understanding of dune dynamics within tide-influenced environment is essentially not efficient. The reason why is that tide oscillation generally prompts a coming and going of large sand volumes
 having a very weak resulting effect on dune evolution. As a consequence, questions concerning dune morphodynamics or stability have to be considered over large periods
of time, making the computation cost expensive.\\
Since many dune fields are present in strong tide region (English Channel, Celtic Sea, Irish Sea, North Sea, etc.) the setting out of methods to tackle dune morphodynamics in tide influenced environments is an important challenge. The aim of this paper is to carry out modeling methods and asymptotical methods for this. More precisely, we focus on linear models for seabed evolution and on 
methods which allow the removal of the explicit
  presence of the tide oscillations from them.

~

As will be seen in section 2,
 for a small parameter $\epsilon$ and constants $a,b$ and $c,$ equation
\begin{equation}\label{eq1}
    \frac{\partial z^\epsilon}{\partial t}-\frac{a}{\epsilon}\nabla\cdot\big((1-b\epsilon \gm)g_{a}(|\gu|)\nabla z^\epsilon\big)=\frac{c}{\epsilon}\nabla\cdot\left((1-b\epsilon \gm)g_{c}(|\gu|)\frac{\gu}{|\gu|}\right),
\end{equation}
is a relevant model for the short-term
 dynamics of dunes. In equation (\ref{eq1}), $z^\epsilon=z^\epsilon(x,t)$ where, for a given constant $T,\,\,t\in[0,T),$ stands for the dimensionless time and $x=(x_{1},x_{2})\in\torus^{2},$ $\torus^{2}$ being the two dimensional torus $\,\mathbb{R}^{2}/\,\mathbb{Z}^{2},$ is the dimensionless position variable, is the dimensionless seabed altitude at time $t$ and in position $x.$ Operators $\nabla$ and $\nabla\cdot$ refer to gradient and divergence. Functions $g_{a}$ and $g_{c}$ are regular on $\R^+$ and satisfy
\begin{equation}\label{eq2}\left\{ \begin{array}{ccc}
g_{a}\geq g_{c}\geq0,\,\,g_{c}(0)=g'_{c}(0)=0,\\
\exists d\geq0,\sup_{u\in\mathbb{R}^{+}}|g_{a}(u)|+\sup_{u\in\mathbb{R}^{+}}|g'_{a}(u)|\leq d,\\
    \sup_{u\in\mathbb{R}^{+}}|g_{c}(u)|+\sup_{u\in\mathbb{R}^{+}}|g'_{c}(u)|\leq d,\\
    \exists U_{thr}\geq0,\,\,\exists G_{thr}>0,\,\,\textrm{such that}\,\, u\geq U_{thr}\Longrightarrow g_{a}(u)\geq G_{thr}.\end{array}\right.
\end{equation}

Fields $\gu$ and $\gm$ are dimensionless water velocity and height. They are given by
\begin{equation}\label{eq3}
\gu(t,x)=\mathcal{U}(t,\frac{t}{\epsilon},x)\quad \gm(t,x)=\mathcal{M}(t,\frac{t}{\epsilon},x),
 \end{equation}
where
\begin{equation}\label{eq4}\left\{ \begin{array}{ccc}
        \ds \mathcal{U}=\mathcal{U}(t,\theta,x)\,\,\textrm {and} \,\,\mathcal{M}=\mathcal{M}(t,\theta,x) \,\,\textrm{are regular functions on}\,\, \mathbb{R}\times\mathbb{R}\times\torus^{2},\\
        \ds\theta\longmapsto(\mathcal{U},\mathcal{M})\,\,\textrm{is periodic of period}\,\, 1,\\
        \ds |\mathcal{U}|,\,\,|\frac{\partial \mathcal{U}}{\partial t}|,\,\,|\frac{\partial \mathcal{U}}{\partial \theta}|,\,\,|\nabla \mathcal{U}|,\,\,|\mathcal{M}|,\,\,|\frac{\partial \mathcal{M}}{\partial t}|,\,\,|\frac{\partial \mathcal{M}}{\partial \theta}|,\,\,|\nabla \mathcal{M}| \,\,\textrm{are bounded by}\,\,d,\\
           \ds \forall (t,\theta,x)\in\mathbb{R}^{+}\times\mathbb{R}\times\torus^{2},\,\,|\mathcal{U}(t,\theta,x)|\leq U_{thr} \Longrightarrow
\hspace{3cm }\\ \hspace{3cm} \ds \frac{\partial\mathcal{U}}{\partial t}=0,\,\,\frac{\partial\mathcal{M}}{\partial t}=0,\,\,\nabla\mathcal{M}(t,\theta,x)=0\,\,\textrm{and}\,\,\nabla\mathcal{U}(t,\theta,x)=0,\\
            \ds\exists \theta_{\alpha}<\theta_{\omega}\in[0,1]\,\,\textrm{such that}\,\, \forall\,\,\theta\in [\theta_{\alpha},\theta_{\omega}]\Longrightarrow|\mathcal{U}(t,\theta,x)|\geq U_{thr} .\end{array}\right.
\end{equation}
\begin{remark}
The last two assumptions
 in (\ref{eq4}) are necessary when $g_{a}$ may vanish. In the case where $g_{a}(u)\geq G_{thr}$ for any $u\geq0,$ then $U_{thr}=0$ and the last two assumptions of (\ref{eq4}) are automatically satisfied by any $\mathcal{U}.$
{\ \hfill \rule{0.5em}{0.5em}}
\end{remark}

~

The following equation, for constants $a,$ $b$ and $c$
\begin{equation}\label{eq5}
\frac{\partial z^\epsilon}{\partial t}-\frac{a}{\epsilon}\nabla\cdot\left((1-b\sqrt{\epsilon}\gm)g_{a}(|\gu|)\nabla z^\epsilon\right)=\frac{c}{\epsilon}\nabla\cdot\left((1-b\sqrt{\epsilon}\gm)g_{c}(|\gu|)\frac{\gu}{|\gu|}\right),
\end{equation}
 with condition (\ref{eq2}) on $g_{a}$ and $g_{c}$ and with  $\gu$ and $\gm$ given by
 \begin{equation}\label{eq6}
 \gu(t,x)=\widetilde{\mathcal{U}}(t,\frac{t}{\sqrt{\epsilon}},\frac{t}{\epsilon},x),\quad \gm(t,x)=\mathcal{M}(t,\frac{t}{\sqrt{\epsilon}},\frac{t}{\epsilon},x),
 \end{equation}
 is a relevant model for mean-term dune dynamics.

 For mathematical reasons, we assume
 \begin{equation}\label{eq7}
 \widetilde{\mathcal{U}}(t,\tau,\theta,x)=\mathcal{U}(t,\theta,x)+ \sqrt{\epsilon}\,\mathcal{U}_{1}(t,\tau,\theta,x),
 \end{equation}
where $\mathcal{U}=\mathcal{U}(t,\theta,x)$ and $\mathcal{U}_{1}=\mathcal{U}_{1}(t,\tau,\theta,x)$ are regular. We also assume  that $\mathcal{M}=\mathcal{M}(t,\tau,\theta,x)$ is regular and
\begin{equation}\label{eq8}\left\{\begin{array}{ccc}
\ds\theta\longmapsto(\mathcal{U},\mathcal{U}_{1},\mathcal{M})\,\,\textrm{is periodic of period 1,}\\
\ds\tau\longmapsto(\mathcal{U}_{1},\mathcal{M})\,\,\textrm{is periodic of period 1,}\\
\ds |\mathcal{U}|,\,\,|\frac{\partial \mathcal{U}}{\partial t}|,\,\,|\frac{\partial \mathcal{U}}{\partial \theta}|,\,\,|\nabla \mathcal{U}|,\,\,|\mathcal{U}_{1}|,\,\,|\frac{\partial \mathcal{U}_{1}}{\partial t}|,\,\,|\frac{\partial \mathcal{U}_{1}}{\partial \theta}|,\,\,|\nabla \mathcal{U}_{1}|,
\hspace{3cm}\\ \hspace{3cm}
\ds |\mathcal{M}|,\,\,|\frac{\partial \mathcal{M}}{\partial t}|,\,\,|\frac{\partial \mathcal{M}}{\partial \theta}|,\,\,|\frac{\partial \mathcal{M}}{\partial \tau}|,\,\,|\nabla \mathcal{M}| \,\,\textrm{are bounded by}\,\,d,\\
            \ds\forall (t,\tau,\theta,x)\in\mathbb{R}^{+}\times\mathbb{R}\times\mathbb{R}\times\torus^{2},\,\,|\widetilde{\mathcal{U}}(t,\theta,x)|\leq U_{thr}\Longrightarrow
 \hspace{2.5cm}\\  \hspace{0.5cm}
 \ds\frac{\partial\widetilde{\mathcal{U}}}{\partial t} (t,\theta,x) =0,\,\,\frac{\partial\widetilde{\mathcal{U}}}{\partial \tau} (t,\theta,x) =0,\,\,\nabla\widetilde{\mathcal{U}}(t,\tau,\theta,x)=0,
 \\ \hspace{2.5cm}
 \ds\frac{\partial\mathcal{M}}{\partial t} (t,\theta,x) =0,\,\,\frac{\partial\mathcal{M}}{\partial \tau} (t,\theta,x) =0\textrm{ and}\,\,\nabla\mathcal{M}(t,\theta,x)=0,\\
            \exists \theta_{\alpha}<\theta_{\omega}\in[0,1]\,\,\textrm{such that}\,\, \forall\,\,\theta\in [\theta_{\alpha},\theta_{\omega}]\Longrightarrow|\mathcal{U}(t,\tau,\theta,x)|\geq U_{thr}.
\end{array}\right.\end{equation}

~

A relevant model for long-term
 dune dynamics is the following equation
\begin{equation}\label{eq81}
\frac{\partial z^\epsilon}{\partial t}-\frac{a}{\epsilon^{2}}\nabla\cdot\left((1-b\epsilon \gm)g_{a}(|\gu|)\nabla z^\epsilon\right)=\frac{c}{\epsilon^{2}}\nabla\cdot\left((1-b\epsilon \gm)g_{c}(|\gu|)\frac{\gu}{|\gu|}\right),
\end{equation}
where $a,$ $b$ and $c$ are constants, where $g_{a}$ and $g_{c}$ satisfy assumption (\ref{eq2}), and where $z$ is defined on the same space as before. It is also relevant to assume
\begin{equation}\label{eq9}
\gu(x,t)=\mathcal{U}(\frac{t}{\epsilon},x)+\epsilon^{2}\mathcal{U}_{2}(t,\frac{t}{\epsilon},x),\quad \gm(t,x)=\mathcal{M}(\frac{t}{\epsilon},x)+\epsilon^{2}\mathcal{M}_{2}(t,\frac{t}{\epsilon},x)
\end{equation}
where $\mathcal{U}=\mathcal{U}(\theta,x),\,\,\mathcal{U}_{2}(t,\theta,x),\,\,\mathcal{M}=\mathcal{M}(\theta,x)$ and $\mathcal{M}_{2}=\mathcal{M}_{2}(t,\theta,x)$ are regular and
\begin{equation}\label{eq10}\left\{\begin{array}{ccc}
\theta\longmapsto(\mathcal{U},\mathcal{U}_{2},\mathcal{M},\mathcal{M}_{2})\,\,\textrm{is periodic of period 1},\\
\ds|\mathcal{U}|,\,\,|\frac{\partial \mathcal{U}}{\partial t}|,\,\,|\frac{\partial \mathcal{U}}{\partial \theta}|,\,\,|\nabla \mathcal{U}|,\,\,|\mathcal{U}_{2}|,\,\,|\frac{\partial \mathcal{U}_{2}}{\partial t}|,\,\,|\frac{\partial \mathcal{U}_{2}}{\partial \theta}|,\,\,|\nabla \mathcal{U}_{2}|,\,\,
\ds |\mathcal{M}|,\,\,|\frac{\partial \mathcal{M}}{\partial t}|,\,\,|\frac{\partial \mathcal{M}}{\partial \theta}|,\,\,|\frac{\partial \mathcal{M}}{\partial \tau}|,\,\,|\nabla \mathcal{M}|, \,\,
\\ \hspace{4.3cm}
\ds |\mathcal{M}_{2}|,\,\,|\frac{\partial \mathcal{M}_{2}}{\partial t}|,\,\,|\frac{\partial \mathcal{M}_{2}}{\partial \theta}|,\,\,|\frac{\partial \mathcal{M}_{2}}{\partial \tau}|,\,\,|\nabla \mathcal{M}_{2}|\textrm{ are bounded by}\,\,d,\\
\ds \forall (t,\theta,x)\in\mathbb{R}^{+}\times\mathbb{R}\times\torus^{2},\,\,|\widetilde{\mathcal{U}}(t,\theta,x)+\epsilon^{2}\mathcal{U}_{2}(t,\theta,x)|\leq U_{thr}\Longrightarrow
\hspace{2cm}\\ \hspace{0.5cm}
\ds \frac{\partial{\mathcal{U}_{2}}}{\partial t} (\theta,x) =0,\,\,\nabla{\mathcal{U}}(\theta,x)=0,\,\,\nabla{\mathcal{U}_{2}}(\theta,x)=0,
\\\hspace{2cm}
\ds\frac{\partial\mathcal{M}_{2}}{\partial t}=0,\,\,\nabla\mathcal{M}(\theta,x)=0,\,\,\nabla\mathcal{M}_{2}(t,\theta,x)=0,\\
\ds\exists \theta_{\alpha}<\theta_{\omega}\in[0,1]\,\,\textrm{such that}\,\, \forall\,\,\theta\in\mathbb{R},\,\,\theta\in [\theta_{\alpha},\theta_{\omega}]\Longrightarrow|\mathcal{U}(\theta,x)+\epsilon^{2}\mathcal{U}_{2}(t,\theta,x)|\geq U_{thr}.
\end{array}\right.\end{equation}

~

Equations (\ref{eq1}), (\ref{eq5}) or (\ref{eq81}) need to be provided with an initial condition
\begin{equation}\label{eq11}
z^{\epsilon}_{|t=0}=z_{0},
\end{equation}
giving the shape of the seabed at the initial time.

~

We will now state the main results of the paper.
The first concerns existence and uniqueness for the short and mean-term models.

\begin{theorem}\label{th1}
For any $T>0,$ any $a>0,$ any real constants $b$ and $c$ and any $\epsilon>0,$ under assumptions (\ref{eq2}) and (\ref{eq3}),(\ref{eq4}) or (\ref{eq7}),(\ref{eq8}), if
\begin{equation}\label{eq12}
z_{0}\,\,\in L^{2}(\torus^{2}),
\end{equation}
there exists a unique function $z^{\epsilon}\in L^{\infty}([0,T),L^{2}(\torus^{2})),$ solution to equation (\ref{eq1}) or (\ref{eq5}) provided with initial condition (\ref{eq11}).\\
Moreover, for any $t\in[0,T],\,\, z^{\epsilon}$ satisfy
\begin{equation}\label{eq12.111}
\|z^{\epsilon}\|_{  L^{\infty}([0,T),L^{2}(\torus^{2}))}\leq\widetilde{\gamma},
\end{equation}
for a constant $\widetilde{\gamma}$ not depending on $\epsilon$ and
\begin{equation}\label{eq12.1}
\frac{\ds d\left(\int_{\torus^{2}} z^{\epsilon}(t,x) \, dx\right)}{dt} =0.
\end{equation}
\end{theorem}
The proof of this theorem is done in section \ref{secExEs}, except equality (\ref{eq12.1})  which is directly gotten by integrating (\ref{eq1}) or (\ref{eq5}) with respect to $x$ over $\torus^2$.
\\
In the previous theorem, $L^{2}(\torus^{2})$ stands for the usual space of square integrable functions defined on the torus $\torus^{2}$ and $L^{\infty}([0,T),L^{2}(\torus^{2}))$ stands for the space of functions mapping $[0,T]$ to $L^{2}(\torus^{2})$ and which are bounded. $\|.\|_{L^{\infty}([0,T),L^{2}(\torus^{2}))}$ stands for the usual norm on this space.
\begin{remark}
As equations (\ref{eq1}) and (\ref{eq5}) are linear, almost parabolic equations, the proof of the existence of $z^{\epsilon}$ over a time interval depending on $\epsilon$ is a straight forward consequence of adaptations of results from Ladyzenskaja, Solonnikov and Ural'Ceva \cite{LSU} or Lions \cite{J.L.L}. But here, since we want to follow an asymptotic process consisting in making $\epsilon\rightarrow0$, we need a time interval which does not depend on $\epsilon.$ Because of the presence of $\frac{1}{\epsilon}$ factor and the fact that the diffusion term may cancel, the proof of theorem \ref{th1} needs
 several steps. In a first step, we prove the existence of a solution, periodic in time and space of a parabolic equation. From this first existence result, we deduce existence of a solution,
 periodic in time and space of an ad-doc degenerate parabolic equation.
\\ Those two results are interesting by themselves and complete the theorem collection in the topic of time and space time-periodic solution to parabolic equation. Concerning this topic, we refer for instance to Barles and Souganidis \cite{Barl}, Berestycki, Hamel and Roques \cite{HBFHLR2,HBFHLR}, Bostan \cite{Bostan}, Hansbo \cite{Hansbo}, Kono \cite{Kono}, Nadin \cite{NadinCras,Nadin}, Namah and Roquejoffre \cite{NaRoqj} and Pardoux \cite{Pard}.\\
Then, having on hand the existence
of the space-time periodic solution to the ad-doc degenerate parabolic equation, we can deduce that the solution $z^{\epsilon}$ which exists on $\epsilon$-dependant time interval, remains close to it. This allows us to deduce a large time existence.
{\ \hfill \rule{0.5em}{0.5em}}
\end{remark}
\begin{remark}
Moreover, notice that
 theorem \ref{th1}, theorems \ref{th3.18} and \ref{th3.19} also complete the theorem collection concerning the topic of large time behavior of parabolic equation (see Barles and Souganidis\cite{Barl}, Da Lio \cite{Dalio}, Norris \cite{Norris}, Park and Tanabe \cite{ParkTan}, Pardoux \cite{Pard}, Petita \cite{Petita} and Tanabe \cite{Tanabe}.
{\ \hfill \rule{0.5em}{0.5em}\\}
\end{remark}
We  now give a result concerning the asymptotic behavior as $\epsilon\longrightarrow0$ of 
the short-term model.
\begin{theorem}\label{thAsyBeh}
For any $T>0,$ under the same assumptions as in theorem \ref{th1}, the solution $z^\epsilon$ to
equation (\ref{eq1}) given by theorem \ref{th1} two-scale converges to a profile
$U\in L^{\infty}([0,T],L^{\infty}_\#(\R,L^2(\torus^{2})))$
 which is the unique solution to
\begin{equation}\label{eqHomIntro}
\frac{\partial U}{\partial\theta}
-\nabla\cdot(\widetilde{\mathcal{A}}\nabla U)=\nabla \cdot\widetilde{\mathcal{C}},
\end{equation}
where $\widetilde{\mathcal{A}}$ and $\widetilde{\mathcal{C}}$ are given by
\begin{equation}\label{eqHomIntroBis}
\widetilde{\mathcal{A}}=a\,g_a(|\mathcal{U}(t,\theta,x)|) \,\, \textrm{and}\,\,\widetilde{\mathcal{C}}=c\,g_c(|\mathcal{U}(t,\theta,x)|)\,\frac{\mathcal{U}(t,\theta,x)}{|\mathcal{U}(t,\theta,x)|}.
\end{equation}
\end{theorem}
In this theorem, $L^{\infty}_\#(\R,L^2(\torus^{2}))$ stands for the space of functions depending on $\theta$
and $x$ mapping $\R$ to $L^{2}(\torus^{2})$ and which are periodic of period 1 with respect to $\theta$
and $L^{\infty}([0,T],L^{\infty}_\#(\R,L^2(\torus^{2})))$ stands for the space of functions mapping $[0,T]$ to $L^{\infty}_\#(\R,L^2(\torus^{2}))$ and which are bounded.
For the definition and results about two-scale convergence we refer to Nguetseng \cite{nguetseng:1989}, Allaire \cite{allaire:1992} and Fr\'enod Raviart and Sonnendr\"{u}cker \cite{FRS:1999}.
{\ \hfill \rule{0.5em}{0.5em}\\}
~\\
Finally, we give a corrector result for the short-term model under restrictive assumptions.
\begin{theorem}\label{th3}
Under the same assumptions as in theorem \ref{th1} and if moreover $U_{thr}=0,$ considering 
function $z^{\epsilon}\in L^{\infty}([0,T),L^{2}(\torus^{2})),$ 
solution to (\ref{eq1}) with initial condition (\ref{eq11}) and function $U^{\epsilon}\in L^{\infty}([0,T],L^{\infty}_\#(\R,L^2(\torus^{2})))$ defined by 
\begin{equation}
U^{\epsilon}(t,x)=U(t,\frac{t}{\epsilon},x),
\end{equation} 
where $U$ is the solution to (\ref{eqHomIntro}), the following estimate is satisfied:
\begin{equation}
\Big\|\frac{z^{\epsilon}-U^{\epsilon}}{\epsilon}\Big\|_{  L^{\infty}([0,T),L^{2}(\torus^{2}))}\,\,\leq\alpha,
\end{equation}
where $\alpha$ is a constant not depending on $\epsilon.$\\
Furthermore 
\begin{equation}
\frac{z^{\epsilon}-U^{\epsilon}}{\epsilon}\quad\textrm{ two-scale converges to a profile}\,\,U^{1}\in L^{\infty}([0,T],L^{\infty}_\#(\R,L^2(\torus^{2}))),
\end{equation} 
which is the unique solution to
\begin{equation}\label{eq5.5}
\frac{\partial U^{1}}{\partial \theta}-\nabla\cdot\left(\widetilde{\mathcal{A}}\nabla U^{1}\right)=\nabla\cdot\widetilde{\mathcal{C}}_{1}+\frac{\partial U}{\partial t}+\nabla\cdot(\widetilde{\mathcal{A}}_{1}\nabla U),
\end{equation}
where $\widetilde{\mathcal{A}}$ and $\widetilde{\mathcal{C}}$ are given by (\ref{eqHomIntroBis}) and where $\widetilde{\mathcal{A}}_{1}$ and $\widetilde{\mathcal{C}}_{1}$ are given by
\begin{equation}\widetilde{\mathcal{A}}_{1}(t,\theta,x)=-ab\mathcal{M}(t,\theta,x)\,g_a(|\mathcal{U}(t,\theta,x)|),\,\, \widetilde{\mathcal{C}}_{1}(t,\theta,x)=-cb\mathcal{M}(t,\theta,x)\,g_c(|\mathcal{U}(t,\theta,x)|)\,\frac{\mathcal{U}(t,\theta,x)}{|U(t,\theta,x)|}.\end{equation}
\end{theorem}
\begin{remark}
Theorems \ref{thAsyBeh} and \ref{th3} state a rigorous version of asymptotic expansion of   
$z^{\epsilon}$:
\begin{gather}
z^{\epsilon}(t,x) = U(t,\frac t\epsilon,x) + \epsilon U_1(t,\frac t\epsilon,x) + \dots \;.
\end{gather}
{\ \hfill \rule{0.5em}{0.5em}}
\end{remark}
{\bf Acknowledgments -} This work is supported by FIRST (Fonds d'Impulsion de la Recherche Scientifique et Technique) du Minist\`ere des Biocarburants des Energies Renouvelables et de  la Recherche Scientifique du S\'en\'egal.

The authors thank Joanna Ropers for proofreading the manuscript.
\section{Modeling}\label{secMod}
\subsection{Sand transport equation}
The equation modeling sand transport
 is the following (see Van Rijn \cite{Rijn1989},
Idier \cite{Idier}):
\begin{equation}\label{M1}
\frac{\partial z}{\partial t}+\frac{1}{1-p}\nabla\cdot q=0.
\end{equation}
In this equation the fields depends on time
  $t\in[0,T)$, for $T>0$, on the horizontal position $x=(x_{1},x_{2})\in \Omega,$ where $\Omega$ is a regular open set of $\mathbb{R}^{2}.$ The field $z=z(t,x)$ is  the height of the seabed
in position $x$ and at time $t$ and $q=q(x,t)$ is the sand volume flow in $x$ and at $t.$ The parameter $p\in[0,1)$ is called sand porosity. Equation (\ref{M1}) has to be coupled with a low linking the sand flow $q$ with the seabed height variation and the velocity of the water near the seabed. Usually, such a law is written
\begin{equation}\label{M2}
    q=q_{f}-|q_{f}|\lambda\nabla z,
\end{equation}
where $q_{f}$ stands for the water velocity  induced sand flow  on a flat seabed and where $|q_{f}|$ stands for its norm. The constant $\lambda$ is the inverse value of the maximum slope of the sediment surface when the water velocity is 0. A generic way to write $q_{f}$ is
\begin{equation}\label{M3}
    q_{f}=\alpha\, \widetilde{\chi}(g(|\gu|)-g(u_{c}))\frac{\gu}{|\gu|},
\end{equation}
where $g$ is a non-negative regular function defined on $\mathbb{R}^{+}$ and
where $\widetilde{\chi}$ is a regular function from $\mathbb{R}$ to $\mathbb{R},$ being 0 on $\mathbb{R}^{-}$ and increasing on $\mathbb{R}^{+}.$ $\gu$ is the water velocity near the seabed,
 $g(u)$ is regular function of $u\in\R^+$ and $u_{c}$ is the threshold under which the water velocity does not make the sand move.\\
Every law encountered in the literature, for instance Meyer-Peter and M\"uller
 \cite{MePetMull} formula, Bagnold and Gadd  formula  (see \cite{Bagnold} and
 \cite{GaddLavSw}) and
 Van Rijn \cite{Rijn1989} formula, is recovered by setting functions $\chi$ and $g.$

In the sequel of the present paper we shall restrict ourselves to
 laws of the Van Rijn type  \cite{Rijn1989}
which consists in writting
\begin{equation}\label{M4}
    q_{f}=\alpha\,\chi\hspace{-2pt}\left(D_{G}(|\tau_{b}|-\tau_{c}) \right)\,\frac{\tau_{b}}{|\tau_{b}|},
\end{equation}
where $\tau_{b}$ is the shear stress density imposed by the water on the seabed. It is linked with $\gu$ by
\begin{equation}\label{M5}
    \tau_{b}=\rho\frac{|\gu|^{2}}{C^{2}}\,\frac{\gu}{|\gu|},
\end{equation}
where $\rho$ is the water density, $C$ is a constant defined by $C=\ln(\frac{12 d}{3D_{G}}),\,\, d$ being the water height above the seabed and $D_{G}$ being the sand speck diameter. The threshold $\tau_{c}$ expresses as
\begin{equation}\label{M6}
    \tau_{c}=\rho\frac{u_{c}^{2}}{C^{2}},
\end{equation}
and $\chi$ is given by
\begin{equation}\label{M7}
\begin{array}{rlc}\chi(\sigma)&=\;\;0& \textrm{ if } \sigma<0,\\
    &=\;\;|\sigma^{3/2}|&\textrm{ if }  \sigma\geq0.\end{array}
\end{equation}
The order of magnitude of constant $\alpha$ is 100.

~

Injecting equation (\ref{M5}) into (\ref{M4}) and (\ref{M2}) we get
\begin{equation}\label{M8}
    q=\alpha\, \chi \hspace{-3pt}\left(D_{G}\,\rho\,\frac{|\gu|^{2}-{u_{c}}^{2}}{C^{2}} \right)
    \left(\frac{\gu}{|\gu|}-\lambda\nabla z\right),
\end{equation}
and equation (\ref{M1}) reads
\begin{equation}\label{M9}
    \frac{\partial z}{\partial t}+\frac{\alpha}{1-p}\nabla\cdot\left[\chi \hspace{-3pt}\left(D_{G}\,\rho\,\frac{|\gu|^{2}-{u_{c}}^{2}}{C^{2}}\right)\left(\frac{\gu}{|\gu|}-\lambda\nabla z\right)\right]=0.
\end{equation}

 \subsection{Scaling}
Now, we will scale (\ref{M9}) to write a dimensionless version of it.
We introduce  a characteristic time $\bar{t}$ and a characteristic length $\bar{L}$ and we define the dimensionless variables $t'$ and $x',$ making $\bar{t}$ and $\bar{L}$ the units by
\begin{equation}\label{M12}
    t=\bar{t}t',\quad x=\bar{L}x'.
\end{equation}
We also define $\bar{z}$ the characteristic height of the dunes and the dimensionless seabed height
\begin{equation}\label{M13}
    z'(t',x')=\frac{1}{\bar{z}}z(\bar{t}t',\bar{L}x').
\end{equation}
Concerning coefficients of equation (\ref{M9}), we introduce $\bar{u}$ the characteristic velocity of the water, we consider the mean water height $H$ and $\bar{M}$ the characteristic height variation due to the tide. Then we define $u'$ being the dimensionless water height variation by
\begin{equation}\label{M14}
    \gu'(t',x')=\frac{1}{\bar{u}}\gu(\bar{t}t',\bar{L}x'),\quad \gm'(t',x')=\frac{1}{\bar{M}}(d(\bar{t}t',\bar{L}x')-H).
\end{equation}
Once those variables and fields are introduced, we first approximate $C,$ taking into account that $\frac{\bar{M}}{H}$ is small.
\begin{equation}\label{M15}
    C=\ln\hspace{-3pt}\left(\frac{4H}{D_{G}}\right)+\ln\hspace{-3pt}\left(1+\frac{\ds \bar{M}}{H}\gm' \right)\simeq
    \ln\hspace{-3pt}\left(\ds\frac{4H}{D_{G}}\right)+\frac{\bar{M}}{H}\gm' .
\end{equation}
From (\ref{M15}) we get
\begin{equation}\label{M16}
    \frac{1}{C^{3}}\simeq\frac{1}{\left(\ln\hspace{-3pt}\left(\ds\frac{4H}{D_{G}}\right)\right)^{3}}\left(1-3\frac{\bar{M}}{H\ln\hspace{-3pt}\left(\ds\frac{4H}{D_{G}}\right)}\gm'\right).
\end{equation}
Since for instance
\begin{equation}\label{M16bis}
    \nabla z(\bar{t}t',\bar{L}x')=\frac{1}{\bar{z}\bar{L}}\nabla'z'(t',x'),
\end{equation}
we get from equation (\ref{M9}) the following equation for $z'$
\begin{equation}\label{M17}
\begin{array}{ccc}
\ds\frac{\partial z'}{\partial t'}-\frac{\lambda}{1-p}\,\alpha\,\frac{\bar{t}\bar{u}^{3}(\rho D_{G})^{3/2}}{\left(\ln(\frac{4H}{D_{G}})\right)^3\bar{L}^{2}}\;\nabla'\cdot
\left(\left(1-3\frac{\bar{M}}{H\ln(\frac{4H}{D_{G}})}\gm'\right)\chi\hspace{-3pt}\left(|\gu'|^{2}-
\frac{{u_{c}}^{2}}{\bar{u}^{2}}\right)\nabla' z'\right)\\~~~~~~~~~~~~~~~~
\ds=\frac{1}{1-p}\,\alpha\,\frac{\bar{t}\bar{u}^{3}(\rho D_{G})^{3/2}}{\left(\ln(\frac{4H}{D_{G}})\right)^3\bar{L}\bar{z}}\;\nabla'\cdot\left(\left(1-3\frac{\bar{M}}{H\ln(\frac{4H}{D_{G}})}\gm'\right)\chi\hspace{-3pt}\left(|\gu'|^{2}-
\frac{{u_{c}}^{2}}{\bar{u}^{2}}\right)\frac{\gu'}{|\gu'|} \right).
\end{array}
\end{equation}
Having this dimensionless model on hand,
 we will now consider several situations in setting the characteristic values for short, mean and long-term
 dune evolution and for small and big sand specks.

~

First, we fix the characteristic sizes which are common
for every situation. Dunes exist within coastal ocean waters over a relatively flat continental shelf, with a water height of about 30 to 50 meters, with tide induced height variations which are
not too strong and with relatively strong tide currents. Then we set
\begin{equation}\label{M18}
    \bar{u}=1\,m/s,\quad H=50\,m,\quad\bar{M}=5\,m.
\end{equation}
Moreover, the order of magnitude of coefficient $\frac{\lambda}{1-p}$ is 1, then we get
\begin{equation}\label{M19}
    \frac{\lambda}{1-p}=1\,\,\textrm{and}\,\,\frac{1}{1-p}=2.
\end{equation}

~

Now we detail the sizes of every characteristic value and of their concerned ratios in equation
 (\ref{M16}), (\ref{M17}) for every situation.

\noindent\textbf{Short-term
 dynamics of dunes made of a small sand specks}\\
Here, we
shall consider that $\bar{t}$ is an observation period of time. We take as $\bar{t}$ the order of magnitude of the smallest period of time during which dunes undergo significant evolution in a tide-submitted environment, i.e.  $\bar{t}=100\, days\sim2400\, hours\sim8.6\; 10^{6}\,s.$ Introducing
$\bar{\omega}$ the main tide frequency, $\bar{t}$ has to be compared with the main tide period
$\frac{1}{\bar{\omega}}\sim13\,hours\sim4.7\;10^{4}\,s.$ This leads to the definition of a small parameter $\epsilon$:
\begin{equation}\label{M20}
    \frac{1}{\bar{t}\bar{\omega}}\sim\frac{1}{200}=\epsilon.
\end{equation}
We consider that the sand speck diameter $D_{G}$ is $0.1mm=10^{-4}\,m.$ According to Flemming \cite{Flemm} and Idier \cite{Idier}, this gives rise to dunes being about 1 meter high, the wave length of which is about 10 meters. Then we set
\begin{equation}\label{M21}
    \bar{z}=1\,m\,\,\textrm{and}\,\,\bar{L}=10\,m.
\end{equation}
We also consider that the critical velocity $u_{c}$ is small compared with $\bar{u}.$ In other words we set
\begin{equation}\label{M21bis}
    \frac{u_{c}^{2}}{\bar{u}^2}=0.
\end{equation}
As the computations of the factors in (\ref{M21}) yields
\begin{equation}\label{M22} \begin{array}{ccc}
\ds\frac{\lambda}{1-p}\alpha\frac{\bar{t}\bar{u}^{3}(\rho D_{G})^{3/2}}{\left(\ln(\frac{4H}{D_{G}})\right)^3\bar{L}^{2}}\sim90\sim\frac{1}{2\epsilon},\\
\ds\frac{\lambda}{1-p}\alpha\frac{\bar{t}\bar{u}^{3}(\rho D_{G})^{3/2}}{\left(\ln(\frac{4H}{D_{G}})\right)^3\bar{L}\bar{z}}\sim1800\sim\frac{10}{\epsilon},\\
\ds\frac{3\bar{M}}{H\ln(\frac{4H}{D_{G}})}\sim2.10^{-2}\sim4\epsilon,
\end{array}\end{equation}
equation (\ref{M17}) reads
\begin{equation}\label{M23}
    \frac{\partial z}{\partial t}-\frac{1}{2\epsilon}\nabla\cdot((1-4\epsilon \gm)|\gu|^{3}\nabla z)
    =\frac{10}{\epsilon}\nabla \cdot((1-4\epsilon \gm)|\gu|^{2}\gu),
 \end{equation}
where we removed the '.\\
Concerning fluid fields $\gu$ and $\gm,$ we assume that they are periodic functions, with modulated amplitude, and of period the tide period. In other words
\begin{equation}\label{M24}
    \gu(x,t)=\mathcal{U}(t,\frac{t}{\epsilon},x),\quad \gm(x,t)=\mathcal{M}(t,\frac{t}{\epsilon},x),
\end{equation}
for functions $\mathcal{U}$ and $\mathcal{M}$ being regular, and such that $\theta\longmapsto(\mathcal{U}(t,\theta,x),\mathcal{M}(t,\theta,x))$ is periodic of period 1, with a null mean value.

Finally, as dunes of the considered kind are, in nature, gathered into dunes fields
it is not completely unrealistic to set equation (\ref{M23}) in a periodic position space.

As matter of the fact, considering  equation (\ref{M23}) is appropriate for the study of short-term
dynamics of dunes made of small sand specks with a mathematical point of view.

~

\noindent\textbf{Short-term dynamics of dunes made of a big sand specks}\\
For this regime, we consider:
\begin{equation}\label{M25}\begin{array}{ccc}
    \ds\bar{t}\sim100\,days\sim2400\,hours\sim8.6\; 10^{6}\,s,\\
    \ds\frac{1}{\bar{\omega}}\sim13 \,hours\sim4.7\;10^{4}\,s\quad D_{G}=5\;10^{-3}\,m\\
    \ds\bar{z}=50\,m,\,\,\bar{L}=300\,m,\,\,u_{c}=\frac{1}{2}\,m/s.
    \end{array}
\end{equation}
Then
\begin{equation}\label{M26} \begin{array}{ccc}
\ds\frac{\lambda}{1-p}\alpha\frac{\bar{t}\bar{u}^{3}(\rho D_{G})^{3/2}}{\left(\ln(\frac{4H}{D_{G}})\right)^3\bar{L}^{2}}\sim90\sim\frac{1}{2\epsilon},\\
\ds\frac{\lambda}{1-p}\alpha\frac{\bar{t}\bar{u}^{3}(\rho D_{G})^{3/2}}{\left(\ln(\frac{4H}{D_{G}})\right)^3\bar{L}\bar{z}}\sim1000\sim\frac{5}{\epsilon},\\
\ds \frac{3\bar{M}}{H\ln(\frac{4H}{D_{G}})}\sim 1.3\;10^{-2}\sim3\epsilon,
\end{array}\end{equation}
equation (\ref{M17}), with the ' removed, gives
\begin{equation}\label{M27}
    \frac{\partial z}{\partial t}-\frac{1}{2\epsilon}\nabla\cdot\left((1-3\epsilon \gm)\chi(|\gu|^{2}-\frac{1}{2})\nabla z\right)\\
    =\frac{5}{\epsilon}\nabla\cdot\left((1-3\epsilon \gm)\chi(|\gu|^{2}-\frac{1}{2})\frac{\gu}{|\gu|}\right).
\end{equation}

~

\noindent\textbf{Mean-term dynamics of dunes made of a small sand specks}\\
By mean-term we mean a period of time of $4.5\, years\sim54\,months\sim1.4\;10^{8}\,s.$
Then, we take $\bar{t}=1.4\;10^{8}\,s,$ which is compared with $\frac{1}{\bar{\omega}}\sim13\, hours\sim4.7\;10^{4}\,s$ giving
\begin{equation}\label{M40}
    \frac{1}{\bar{t}\bar{\omega}}\sim\frac{1}{3000}=\epsilon.
\end{equation}
We also consider a second tide period which is the time for the earth, the moon and the sun to recover approximately the same relative positions. This period of time $\frac{1}{\bar{\omega}_{c}}$ is about one month. So we have
\begin{equation}\label{M41}
    \frac{1}{\bar{t}\bar{\omega}_{c}}\sim\frac{1}{54}\sim\sqrt{\epsilon}.
\end{equation}
We also take
 $D_{G}=5\;10^{-5}\,m$ and
\begin{equation}\label{M42}
    \bar{z}=1\,m,\,\,\bar{L}=10\,m,\,\,{u_{c}}=0.
\end{equation}
Computing the coefficients in equation (\ref{M17}) gives
\begin{equation}\label{M43}
    \frac{\partial z}{\partial t}-\frac{1}{\epsilon}\nabla\cdot((1-\sqrt{\epsilon}\, \gm)|\gu|^{3}\,\nabla z)\\
    =\frac{20}{\epsilon}\nabla\cdot((1-\sqrt{\epsilon}\, \gm)|\gu|^{2}\gu).
\end{equation}
As was previously seen, it
 is reasonable to set this equation in a periodic domain and concerning the fluid fields we consider
\begin{equation}\label{M44}
    \gu(t,x)=\widetilde{\mathcal{U}} (t,\frac{t}{\sqrt{\epsilon}},\frac{t}{\epsilon}),\,\,\gm(t,x)=
    \mathcal{M}(t,\frac{t}{\sqrt{\epsilon}},\frac{t}{\epsilon}).
\end{equation}
to take into account the two tide periods under consideration. In (\ref{M44}) we take $\mathcal{U}$ and $\mathcal{M}$ as regular functions such that
\begin{equation}\label{M45}\left\{\begin{array}{ccc}
\tau\longmapsto(\widetilde{\mathcal{U}}(t,\tau,\theta,x),\mathcal{M}(t,\tau,\theta,x))\\
    \theta\longmapsto(\widetilde{\mathcal{U}}(t,\tau,\theta,x),\mathcal{M}(t,\tau,\theta,x))
\end{array}\right.\end{equation}
are periodic of period 1.

~

\noindent\textbf{Long-term dynamics of dunes made of small sand specks}\\
We take here $\bar{t}\sim16\,years\sim1.4\;10^{5}\,hours \sim5\;10^{9}s$. We compare this period of time  with the  second tide period $\frac{1}{\bar{\omega}_{c}}\sim1\,month\sim2.6\;10^{6}s.$ Then, we
 define $\epsilon$ by
\begin{equation}\label{M50}
    \frac{1}{\bar{t}\bar{\omega}_{c}}\sim\frac{1}{192}=\epsilon.
\end{equation}
We set
\begin{equation}\label{M51}
    D_{G}=7.10^{-5}\,m,\,\, \bar{z}=1\,m,\,\,\bar{L}=10\,m,\,\,u_{c}=0 \, m/s.
\end{equation}
with those values equation (\ref{M17}) yields
\begin{equation}\label{M52}
    \frac{\partial z}{\partial t}+\frac{1}{\epsilon^{2}}\nabla\cdot((1-4\epsilon \gm)|\gu|^{3}\,\nabla z)\\
    =\frac{20}{\epsilon^{2}}\nabla\cdot((1-4\epsilon \gm)|\gu|^{2}\gu).
\end{equation}
As, at the second tide period scale the tide phenomena may almost be considered as really periodic we set
\begin{equation}\label{53}\begin{array}{ccc}
    \ds\gu(x,t)=\mathcal{U}(\frac{t}{\epsilon},x)+\epsilon^{2}\mathcal{U}_{2}(t,\frac{t}{\epsilon},x),\vspace{2pt}\\
    \ds\gm(x,t)=\mathcal{M}(\frac{t}{\epsilon},x)+\epsilon^{2}\mathcal{M}_{2}(t,\frac{t}{\epsilon},x),
    \end{array}
\end{equation}
where $\mathcal{U},\,\,\mathcal{U}_{2},\,\,\mathcal{M}\,\,\mathcal{M}_{2}$ are regular functions such that $\theta\longmapsto(\mathcal{U}(\theta,x),\mathcal{U}_{2}(t,\theta,x),\mathcal{M}(\theta,x),\mathcal{M}_{2}(t,\theta,x))$
is periodic of period 1 and such that
\begin{equation}\label{M54}
    \int_{0}^{1}\mathcal{U}(\theta,x)d\theta=0,
\end{equation}
\begin{equation}\label{M54.111}
    \int_{0}^{1}\mathcal{M}(\theta,x)d\theta=0.
\end{equation}
\section{Existence and estimates, proof of theorem \ref{th1}}\label{secExEs}
Defining
\begin{equation}\label{3.1}
    \mathcal{A}^{\epsilon}(t,x)=\widetilde{\mathcal{A}}_{\epsilon}(t,\frac{t}{\sqrt{\epsilon}},\frac{t}{\epsilon},x),
\end{equation}
where
\begin{equation}\label{3.2}\widetilde{\mathcal{A}}_{\epsilon}(t,\tau,\theta,x)=a\Big(1-b\sqrt{\epsilon}\,\mathcal{M}
(t,\tau,\theta,x)\Big)\,
g_a\Big(|\mathcal{U}(t,\theta,x)+\sqrt{\epsilon}\,\mathcal{U}_{1}(t,\tau,\theta,x)|\Big),
\end{equation}
\begin{equation}\label{3.3}
    \mathcal{C}^{\epsilon}(t,x)=\widetilde{\mathcal{C}}_{\epsilon}(t,\frac{t}{\sqrt{\epsilon}},\frac{t}{\epsilon},x),
\end{equation}
and where
\begin{eqnarray}\label{3.4}{ \widetilde{\mathcal{C}}_{\epsilon}(t,\tau,\theta,x)=c\Big(1-b\sqrt{\epsilon}\mathcal{M}(t,\tau,\theta,x)\Big)\;g_c\Big(|\mathcal{U}(t,\theta,x)
+\sqrt{\epsilon}\,\mathcal{U}_{1}(t,\tau,\theta,x)|\Big) {} }~~~~~
\nonumber\\
{}
\times\frac{\mathcal{U}(t,\theta,x)+\sqrt{\epsilon}\,\mathcal{U}(t,\tau,\theta,x)}{|\mathcal{U}(t,\theta,x)+\sqrt{\epsilon}\,\mathcal{U}(t,\tau,\theta,x)|},
\end{eqnarray}
equation (\ref{eq5}), (\ref{eq11}) with assumptions (\ref{eq6}) and (\ref{eq7}) reads
\begin{equation}\label{3.5}\left\{\begin{array}{cc}
    \ds \frac{\partial z^{\epsilon}}{\partial t}-\frac{1}{\epsilon}\nabla\cdot(\mathcal{A}^{\epsilon}\nabla z^{\epsilon})=\frac{1}{\epsilon}\nabla\cdot\mathcal{C}^{\epsilon},\\
    \ds z^{\epsilon}_{|t=0}=z_{0}.
\end {array}\right.\end{equation}

In the same way, setting
\begin{equation}\label{3.6}\widetilde{\mathcal{A}}_{\epsilon}(t,\tau,\theta,x)=\widetilde{\mathcal{A}}_{\epsilon}(t,\theta,x)
=a(1-b\epsilon\mathcal{M}(t,\theta,x))\,
g_a(|\mathcal{U}(t,\theta,x)|),
\end{equation}
and
\begin{equation}\label{3.7} \widetilde{\mathcal{C}}_{\epsilon}(t,\tau,\theta,x)=\widetilde{\mathcal{C}}_{\epsilon}(t,\theta,x)=
c(1-b\epsilon\mathcal{M}(t,\theta,x))\,g_c(|\mathcal{U}(t,\theta,x)|)\,
\frac{\mathcal{U}(t,\theta,x)}{|\mathcal{U}(t,\theta,x)|},
\end{equation}
and defining $\mathcal{A}^{\epsilon}$ and $\mathcal{C}^{\epsilon}$ from $\widetilde{\mathcal{A}}_{\epsilon}$ and $\widetilde{\mathcal{C}}_{\epsilon}$ by (\ref{3.1}) and (\ref{3.3}), we may deduce that equation (\ref{eq1}), with assumption (\ref{eq3}), can be set in the form (\ref{3.5}).

~

From assumptions (\ref{eq2}) and (\ref{eq4}) or (\ref{eq2}) and (\ref{eq8}), $\widetilde{\mathcal{A}}_{\epsilon}$ defined by (\ref{3.2}) or (\ref{3.6}) and $\widetilde{\mathcal{C}}_{\epsilon}$ defined by (\ref{3.4}) or (\ref{3.7})
satisfy the following properties
\begin{equation}\label{3.12}\begin{array}{ccc}
    \ds|\widetilde{\mathcal{A}}_{\epsilon}|\leq\gamma,\,\,|\widetilde{\mathcal{C}}_{\epsilon}|\leq\gamma,\,\,\left|\frac{\partial\widetilde{\mathcal{A}}_{\epsilon}}{\partial t}\right|\leq\gamma,\,\,\left|\frac{\partial\widetilde{\mathcal{C}}_{\epsilon}}{\partial t}\right|\leq\gamma,\\
    \ds\left|\frac{\partial\widetilde{\mathcal{A}}_{\epsilon}}{\partial \theta}\right|\leq\gamma,\,\,
    \left|\frac{\partial\widetilde{\mathcal{C}}_{\epsilon}}{\partial \theta}\right|\leq\gamma,\,\,|\nabla\widetilde{\mathcal{A}}_{\epsilon}|\leq\gamma,\,\,|\nabla\cdot\widetilde{\mathcal{C}}_{\epsilon}|\leq\gamma,\\
    \ds \left|\frac{\partial\nabla\widetilde{\mathcal{A}}_{\epsilon}}{\partial t}\right|\leq\gamma,
    \,\,\ds \left|\frac{\partial\nabla\cdot\widetilde{\mathcal{C}}_{\epsilon}}{\partial t}\right|\leq\gamma ,
\end{array}\end{equation}
\begin{equation}\label{3.13}
\left|\frac{\partial\widetilde{\mathcal{A}}_{\epsilon}}{\partial\tau}\right|\leq\sqrt{\epsilon}\gamma,\,\,
\left |\frac{\partial\widetilde{\mathcal{C}}_{\epsilon}}{\partial\tau} \right |\leq\sqrt{\epsilon}\gamma,\,\,
\left |\frac{\partial\nabla\widetilde{\mathcal{A}}_{\epsilon}}{\partial\tau} \right |\leq\sqrt{\epsilon}\gamma,
\end{equation}
on $\mathbb{R}^{+}\times\mathbb{R}\times\mathbb{R}\times\torus^{2},$ for a constant $\gamma$ depending only on $a,b,c$ and $d$ and not on $\epsilon.$ \\
Concerning (\ref{3.13}) in the case where $\widetilde{\mathcal{A}}_{\epsilon}$ and $\widetilde{\mathcal{C}}_{\epsilon}$ are defined by (\ref{3.6}) and (\ref{3.7}), it reduces to
\begin{equation}\label{3.14}
    \frac{\partial\widetilde{\mathcal{A}}_{\epsilon}}{\partial\tau}=\frac{\partial\widetilde{\mathcal{A}}_{\epsilon}}{\partial\tau}=0.
\end{equation}
Moreover, for every $\epsilon,$ $0\leq\epsilon\leq1,$ $\widetilde{\mathcal{A}}_{\epsilon}\geq0,$
\begin{equation}\label{3.15}\left\{\begin{array}{ccc}
    \tau\longmapsto(\widetilde{\mathcal{A}}_{\epsilon},\widetilde{\mathcal{C}}_{\epsilon})\,\,\textrm{is periodic of period}\,\,1,\\
    \theta\longmapsto(\widetilde{\mathcal{A}}_{\epsilon},\widetilde{\mathcal{C}}_{\epsilon})\,\,\textrm{is periodic of period}\,\,1,\end{array}\right.
\end{equation}
and there exists a constant $\widetilde{G}_{thr}$ depending only on $a,b,d$ and $G_{thr}$ and two numbers $\theta_{\alpha}$ and $\theta_{\omega}$ in $[0,1],\,\,\theta_{\alpha}<\theta_{\omega},$ such that
\begin{equation}\label{3.16}
    \widetilde{\mathcal{A}}_{\epsilon}(t,\tau,\theta,x)\geq\widetilde{G}_{thr},
\end{equation}
for every $t\in\mathbb{R},$ $\tau\in\mathbb{R},\,\,x\in\torus^{2}$ and $\theta\in[\theta_{\alpha},\theta_{\omega}]$  and such that $\forall (t,\tau,\theta,x)\in \mathbb{R}^{+}\times\mathbb{R}\times\mathbb{R}\times\torus^{2}$
\begin{equation}\label{3.17}
\widetilde{\mathcal{A}}_{\epsilon}(t,\tau,\theta,x)\leq\widetilde{G}_{thr}\Longrightarrow
\left\{\begin{array}{ccc}
\ds\frac{\partial\widetilde{\mathcal{A}}_{\epsilon}}{\partial t}(t,\tau,\theta,x)=0,\,\,\frac{\partial\widetilde{\mathcal{A}}_{\epsilon}}{\partial \tau}(t,\tau,\theta,x)=0,\,\,\nabla\widetilde{\mathcal{A}}_{\epsilon}(t,\tau,\theta,x)=0, \vspace{3pt}\\
\ds\frac{\partial\widetilde{\mathcal{C}}_{\epsilon}}{\partial t}(t,\tau,\theta,x)=0,\,\,\frac{\partial\widetilde{\mathcal{C}}_{\epsilon}}{\partial \tau}(t,\tau,\theta,x)=0,\,\,\nabla\cdot\widetilde{\mathcal{C}}_{\epsilon}(t,\tau,\theta,x)=0.\\
\end{array}\right.\end{equation}

We also have the following inequalities
\begin{equation}\label{3.171}
    |\widetilde{\mathcal{C}}_{\epsilon}|\leq\gamma|\widetilde{\mathcal{A}}_{\epsilon}|,
\end{equation}
\begin{equation}\label{3.172}
    |\widetilde{\mathcal{C}}_{\epsilon}|^{2}\leq\gamma|\widetilde{\mathcal{A}}_{\epsilon}|,
\end{equation}
\begin{equation}\label{3.173}
    |\nabla\widetilde{\mathcal{A}}_{\epsilon}|\leq\gamma|\widetilde{\mathcal{A}}_{\epsilon}|,
\end{equation}
\begin{equation}\label{3.174}
    \Big|\frac{\partial\widetilde{\mathcal{A}}_{\epsilon}}{\partial t}\Big|\leq\gamma|\widetilde{\mathcal{A}}_{\epsilon}|,
\end{equation}
\begin{equation}\label{3.175}
    \Big|\frac{\partial(\nabla\widetilde{\mathcal{A}}_{\epsilon})}{\partial t}\Big|^{2}\leq\gamma|\widetilde{\mathcal{A}}_{\epsilon}|,
\end{equation}
\begin{equation}\label{3.176}
    \Big|\frac{\partial\widetilde{\mathcal{A}}_{\epsilon}}{\partial\tau}\Big|^{2}
    \leq\epsilon\gamma|\widetilde{\mathcal{A}}_{\epsilon}|,
\end{equation}
\begin{equation}\label{3.177}
    \Big|\frac{\partial\nabla\widetilde{\mathcal{A}}_{\epsilon}}{\partial\tau}\Big|^{2}\leq\epsilon\gamma|\widetilde{\mathcal{A}}_{\epsilon}|,
\end{equation}
\begin{equation}\label{3.178}
    \Big|\nabla \cdot\widetilde{\mathcal{C}}_{\epsilon}\Big|\leq\gamma|\widetilde{\mathcal{A}}_{\epsilon}|,
\end{equation}
\begin{equation}\label{3.179}
    \Big|\frac{\partial\widetilde{\mathcal{C}}_{\epsilon}}{\partial t}\Big|\leq\gamma|\widetilde{\mathcal{A}}_{\epsilon}|,
\end{equation}
\begin{equation}\label{3.1710} 
    \Big|\frac{\partial\widetilde{\mathcal{C}}_{\epsilon}}{\partial t}\Big|^{2}\leq\gamma^2|\widetilde{\mathcal{A}}_{\epsilon}|,
\end{equation}
possibly changing the value of $\gamma$ and making it also depend on $\widetilde{G}_{thr}.$
\\
Inequality (\ref{3.171}) is a direct consequence of (\ref{eq2}).
Since for small values of $|\widetilde{\mathcal{C}}_{\epsilon}|,|\widetilde{\mathcal{C}}_{\epsilon}|^{2}\leq|\widetilde{\mathcal{C}}_{\epsilon}|$ and since $\widetilde{\mathcal{C}}_{\epsilon}$ and $\widetilde{\mathcal{A}}_{\epsilon}$ are bounded, inequality (\ref{3.172}) follows from (\ref{3.171}).
When $|\widetilde{\mathcal{A}}_{\epsilon}|\leq\widetilde{G}_{thr}$, then $\nabla\widetilde{\mathcal{A}}_{\epsilon}=0$. Hence (\ref{3.173}) is realized. When $|\widetilde{\mathcal{A}}_{\epsilon}|\geq\widetilde{G}_{thr}$, since $\widetilde{\mathcal{A}}_{\epsilon}$ and $\nabla\widetilde{\mathcal{A}}_{\epsilon}$ are bounded, (\ref{3.173}) is obviously realized. Hence (\ref{3.173}) is true. With a similar argument (\ref{3.174})-(\ref{3.177}) may be obtained.
In order to obtain (\ref{3.178}), we just have to notice that, when  $|\widetilde{\mathcal{A}}_{\epsilon}|\leq \widetilde G_{thr}$,  $\nabla\cdot\widetilde{\mathcal{C}}_{\epsilon}=0.$ Hence we can give the same argument as above. In the same manner, the last two inequalities
 may be obtained.

~

We now consider for a positive small parameter $\nu$ the following regularization of (\ref{3.5})
\begin{equation}\label{3.20}\left\{ \begin{array}{ccc}
\ds \frac{\partial z^{\epsilon,\nu}}{\partial t}-\frac{1}{\epsilon}\nabla\cdot((\mathcal{A}^{\epsilon}+\nu)\nabla z^{\epsilon,\nu})=\frac{1}{\epsilon}\nabla\cdot\mathcal{C}^{\epsilon}\\
z^{\epsilon,\nu}_{|t=0}=z_{0}.\end{array}\right.
    \end{equation}
Denoting by $\|\cdot\|_{2}$ and $\|\cdot\|_{\infty}$ the usual norms of spaces $L^{2}(\torus^{2})$ and $L^{\infty}(\torus^{2}),$ applying the energy estimate and the maximum principle, (see for instance Lazyzenskaja, Solonnikov and Ural'Ceva\cite{LSU} or Lions\cite{J.L.L}) we  can obtain the following lemma.
\begin{lemma}
For any $T>0,$ if $z_{0}\in L^{2}\cap L^{\infty}(\torus^{2})$  under assumptions (\ref{3.15}), then for any $\epsilon>0$ and $\nu>0,$ there exists a unique solution $z^{\epsilon,\nu}\in L^{\infty}([0,T);L^{2}\cap L^{\infty}(\torus^{2}))$ to (\ref{3.20}). Moreover it satisfies
\begin{equation}\label{3.22}
    \|z^{\epsilon,\nu}\|_{2} + \|z^{\epsilon,\nu}\|_{\infty}\leq\frac{\gamma_{1}}{\epsilon},
\end{equation}
for a constant $\gamma_{1}$ depending only on $\gamma$ and $\|z_{0}\|_{2} + \|z_{0}\|_{\infty}.$
 \end{lemma}

As estimate (\ref{3.22}) depends on $\nu,$ letting $\nu$ go towards 0 we obtain the following corollary.
\begin{corollary} \label{corEstLoc}
For any $T>0,$ if $z_{0}\in  L^{2}\cap L^{\infty}(\torus^{2}),$ and under assumptions (\ref{3.15}), then for any $\epsilon>0,$ there exists a unique solution $z^{\epsilon}\in L^{\infty}([0,T);L^{2}\cap L^{\infty}(\torus^{2}))$
to (\ref{3.5}). Moreover it satisfies
\begin{equation}\label{3.24}
    \|z^{\epsilon}\|_{2} + \|z^{\epsilon}\|_{\infty}\leq\frac{\gamma_{1}}{\epsilon}.
\end{equation}
\end{corollary}
The uniqueness of the solution to (\ref{3.5}) is a direct consequence of the linearity of this equation.

~

~

As we want to study the asymptotic behavior of $z^{\epsilon}$ as $\epsilon$ goes to 0, estimates  (\ref{3.24}) and (\ref{3.22}) are not enough. We need estimates which do not depend on $\epsilon.$ For this, we first consider the following problems, which consists in finding $\mathcal{S}^{\nu}=\mathcal{S}^{\nu}(t,\tau,\theta,x)$ and
$\mathcal{S}^{\nu}_{\mu}=\mathcal{S}^{\nu}_{\mu}(t,\tau,\theta,x)$ being periodic of period 1 with respect to $\theta,$ solutions to
\begin{equation}\label{3.26}
    \frac{\partial\mathcal{S}^{\nu}}{\partial\theta}-\nabla\cdot((\widetilde{\mathcal{A}}_{\epsilon}
    (t,\tau,\cdot,\cdot)+\nu)\nabla\mathcal{S}^{\nu})=\nabla\cdot\widetilde{\mathcal{C}}_{\epsilon}(t,\tau,\cdot,\cdot),
\end{equation}
and
\begin{equation}\label{3.27}
    \mu\mathcal{S}^{\nu}_{\mu}+\frac{\partial\mathcal{S}^{\nu}_{\mu}}{\partial\theta}
    -\nabla\cdot((\widetilde{\mathcal{A}}_{\epsilon}(t,\tau,\cdot,\cdot)+\nu)\nabla\mathcal{S}^{\nu}_{\mu}))=
    \nabla\cdot\widetilde{\mathcal{C}}_{\epsilon}(t,\tau,\cdot,\cdot).
\end{equation}
In equations (\ref{3.26}) and (\ref{3.27}) $t$ and $\tau$ are only parameters.\\
The method to get the desired estimates which do not depend on $\epsilon$ is shared in several steps. In the first, we set out the existence 
of periodic solution  $\mathcal{S}^{\nu}_{\mu}$ to equation  (\ref{3.27}). We, moreover, set
 out that sequence $\mathcal{S}^{\nu}_{\mu}$ is bounded independently of $\mu$ and $\epsilon.$
We also show that $\mathcal{S}^{\nu}_{\mu}$ is differentiable with respect to $t$ and $\tau.$ In a second step, letting $\mu$ go to 0, we get existence  of $\mathcal{S}^{\nu},$  with the same properties as $\mathcal{S}^{\nu}_{\mu}.$ The third step consists in finding estimates on $\mathcal{S}^{\nu},\,\,\frac{\partial\mathcal{S}^{\nu}}{\partial t}$ and $\frac{\partial\mathcal{S}^{\nu}}{\partial \tau}$ which are independent of $\nu$ to be able to make the process $\nu \longrightarrow0$ and to obtain the existence of a solution $\mathcal{S}$ to (\ref{3.26}) with $\nu=0$ and consequently a periodic solution $\zep$ to an equation close to (\ref{3.5}) in a fourth step. The fifth step consists in noticing that the solution $z^{\epsilon}$ of (\ref{3.5}) is not far from $\zep$.

The framework of periodic solutions to parabolic equations is widely studied in both linear and nonlinear cases. We refer for instance to
Barles\cite{Barl},
Berestycki-Hamel and Roques \cite{HBFHLR,HBFHLR2},
Bostan \cite{Bostan}, Hansbo \cite{Hansbo}, Kono \cite{Kono}, Nadin \cite{NadinCras,Nadin}, Namah and Roquejoffre \cite{NaRoqj} and Pardoux \cite{Pard} for a revue of the results on this topic that our result (see theorem \ref{th3.5} and \ref{th3.12}) completes. Inspired by ideas that may be found in those references, concerning equation (\ref{3.27}) we can state the following theorem.
\begin{theorem}\label{th3.5}
Under assumptions (\ref{3.12}), (\ref{3.13}) and (\ref{3.15}), for any $\mu>0$ and any $\nu>0,$ there exists a unique $\mathcal{S}_{\mu}^{\nu}=\mathcal{S}_{\mu}^{\nu}(t,\tau,\theta,x)\in\mathcal{C}^{0}\cap L^{2}(\mathbb{R}\times\torus^{2})$, periodic of period 1 with respect to $\theta,$  solution to (\ref{3.27})
and regular with respect to the parameters $t$ and $\tau$. Moreover, there exists a constant $\gamma_{3}$, which depends only on $\gamma$ and $\nu$ such that
\begin{equation}\label{3.280}
    \sup_{\theta\in\mathbb{R}}\left|\int_{\torus^{2}}\mathcal{S}_{\mu}^{\nu}(\theta,x)dx\right|=0,
\end{equation}
\begin{equation}\label{3.281}
    \left\|\frac{\partial\mathcal{S}_{\mu}^{\nu}}{\partial \theta}\right\|_{L^{2}_{\#}(\mathbb{R}\times\torus^{2})}\leq\gamma_{3},
\end{equation}
\begin{equation}\label{3.282}
    \|\nabla\mathcal{S}_{\mu}^{\nu}\|_{L^{\infty}_{\#}(\mathbb{R},L^{2}(\torus^{2})}\leq\gamma_{3},
\end{equation}
\begin{equation}\label{3.283}
    \|\mathcal{S}_{\mu}^{\nu}\|_{L^{\infty}_{\#}(\mathbb{R},L^{2}(\torus^{2})}\leq\gamma_{3}.
\end{equation}
The following estimates with respect to the parameters $t$ and $\tau$ are also true
\begin{equation}\label{3.29}
    \left\|\frac{\partial\mathcal{S}_{\mu}^{\nu}}{\partial t}\right\|_{L^{\infty}_{\#}(\mathbb{R},L^{2}(\torus^{2}))}\leq\gamma_{3},\qquad \left\|\frac{\partial\mathcal{S}_{\mu}^{\nu}}{\partial \tau}\right\|_{L^{\infty}_{\#}(\mathbb{R},L^{2}(\torus^{2}))}\leq\sqrt{\epsilon}\gamma_{3}.
\end{equation}
\end{theorem}
In the above theorem, the norms are defined by
\begin{gather}
\|f\|^{2}_{L^{2}_{\#}(\mathbb{R}\times\torus^{2})}=\int_{0}^{1}\int_{\torus^{2}}f^{2}\,dxd\theta,\\
\|f\|^{2}_{L^{\infty}_{\#}(\mathbb{R},L^{2}(\torus^{2}))}=\sup_{\theta\in[0,1]}\int_{\torus^{2}}f^{2}\,dx.
\end{gather}
\begin{proof}{{\bf of theorem \ref{th3.5}.}}
The point of departure to prove theorem \ref{th3.5}
consists in considering, for $\xi\in L^{2}(\torus^{2}),$ the solution $\xi_{\mu}^{\nu}$ to
\begin{equation}\label{3.30}\left\{\begin{array}{ccc}
    \ds\mu\xi_{\mu}^{\nu}+\frac{\partial\xi_{\mu}^{\nu}}{\partial\theta}-\nabla\cdot((\widetilde{\mathcal{A}}_{\epsilon}+\nu)\nabla\xi_{\mu}^{\nu})=\nabla\cdot\widetilde{\mathcal{C}}_{\epsilon},\\
\ds{\xi_{\mu}^{\nu}}_{|\theta=0}=\xi,
\end{array}\right.\end{equation}
whose existence and uniqueness on any finite interval is a direct consequence of Ladyzenskaja, Solonnikov and Ural'Ceva \cite{LSU} or Lions \cite{J.L.L}.
We also consider the application
\begin{equation}\label{3.31}
    \square:L^{2}(\torus^{2})\longrightarrow L^{2}(\torus^{2}),\quad
    \xi\longmapsto\xi_{\mu}^{\nu}(1,\cdot),
\end{equation}
and for it we can prove the following lemma.
\begin{lemma} \label{lem3.6}For any $\mu>0,\nu>0$ and $\epsilon>0,$ $\square$ is a strict contraction
\end{lemma}
\begin{proof}{\bf of lemma \ref{lem3.6}.}
For any $\xi\in L^{2}(\torus^{2})$ and any $\widetilde{\xi}\in L^{2}(\torus^{2}),$ we consider $\xi_{\mu}^{\nu}$ and $\widetilde{\xi}_{\mu}^{\nu}$  the solutions of (\ref{3.30}) associated with initial conditions $\xi$ and $\widetilde{\xi}.$ It is obvious to obtain that $\xi_{\mu}^{\nu}-\widetilde{\xi}_{\mu}^{\nu}$ is solution to
\begin{equation}\label{3.32}
    \mu(\xi_{\mu}^{\nu}-\widetilde{\xi}_{\mu}^{\nu})+\frac{\partial(\xi_{\mu}^{\nu}-\widetilde{\xi}_{\mu}^{\nu})}{\partial\theta}-\nabla\cdot\left((\widetilde{\mathcal{A}}_{\epsilon}+\nu)\nabla(\xi_{\mu}^{\nu}-\widetilde{\xi}_{\mu}^{\nu})\right)=0,
\end{equation}
which multiplied by $\xi_{\mu}^{\nu}-\widetilde{\xi}_{\mu}^{\nu}$  and integrated on $\torus^{2}$ yields.
\begin{equation}\label{3.33}
    \mu\|\xi_{\mu}^{\nu}-\widetilde{\xi}_{\mu}^{\nu}\|_{2}^{2}
    + \frac12 \frac{\partial\left(\|\xi_{\mu}^{\nu}-\widetilde{\xi}_{\mu}^{\nu} \|_{2}^{2}\right)}{\partial\theta}
    +\int_{\torus^{2}}(\widetilde{\mathcal{A}}_{\epsilon}+\nu)|\nabla(\xi_{\mu}^{\nu}-\widetilde{\xi}_{\mu}^{\nu}) |^{2}dx=0.
\end{equation}
Since $\widetilde{\mathcal{A}}_{\epsilon}+\nu>0,$ we can deduce that
\begin{equation}\label{3.34}
    \|\xi_{\mu}^{\nu}(1,\cdot)-\widetilde{\xi}_{\mu}^{\nu}(1,\cdot)\|_{2}^{2}\leq e^{-2\mu}\|\xi-\widetilde{\xi}\|_{2}^{2},
\end{equation}
giving the lemma.
\end{proof}\\
From lemma \ref{lem3.6} we deduce that there exists a unique function $\zeta\in L^{2}(\torus^{2})$ such that $\square(\zeta)=\zeta.$ Since $\widetilde{\mathcal{A}}_{\epsilon}$ and $\widetilde{\mathcal{C}}_{\epsilon}$ are periodic with respect to $\theta$ we deduce that the sought periodic solution $\mathcal{S}_{\mu}^{\nu}$ of (\ref{3.27}) is nothing but the solution of (\ref{3.30}) associated with initial condition $\zeta $
such that $\square(\zeta)= \zeta$. Hence we proved the existence of  $\mathcal{S}_{\mu}^{\nu}$ claimed in theorem \ref{th3.5}.
The fact is that $\mathcal{S}_{\mu}^{\nu}$ is continuous comes from the fact that it is a solution of a parabolic equation with regular coefficients. We now turn to the properties of the function $\mathcal{S}_{\mu}^{\nu}.$
\begin{lemma}\label{lem3.7}For any $\mu>0,\,\,\nu>0$ and $\epsilon>0$ the solution to (\ref{3.27}) satisfies property (\ref{3.280}).
\end{lemma}
\begin{proof}{\bf of lemma \ref{lem3.7}.} Integrating (\ref{3.27}) over $\torus^{2}$ gives
\begin{equation}\label{3.36}
    \mu\int_{\torus^{2}}\mathcal{S}_{\mu}^{\nu}dx+\frac{\ds d\left(\int_{\torus^{2}}\mathcal{S}_{\mu}^{\nu}dx\right)}{d\theta}\,\,=0,
\end{equation}
which leads,
\begin{equation}\label{3.37}
   \int_{\torus^{2}}\mathcal{S}_{\mu}^{\nu}(\theta,x)dx=\int_{\torus^{2}}\mathcal{S}_{\mu}^{\nu}(\tilde{\theta},x)dx \; e^{-\mu(\theta-\tilde{\theta})}.
\end{equation}
Since $\mathcal{S}_{\mu}^{\nu}$ is periodic of period 1 with respect to $\theta$, $\int_{\torus^{2}}\mathcal{S}_{\mu}^{\nu}dx$ is also periodic of period 1. Hence the only possibility to satisfy (\ref{3.37}) for $\int_{\torus^{2}}\mathcal{S}_{\mu}^{\nu}dx$ is to be 0.
This ends the proof of lemma \ref{lem3.7}
\end{proof}
\begin{lemma}\label{lem3.8}
For any $\mu>0,\,\,\nu>0$ and $\epsilon>0$ estimate
\begin{equation}\label{3.371}
\|\nabla\mathcal{S}^{\nu}_{\mu}\|_{L^{2}_{\#}(\mathbb{R}\times\torus^{2})}\leq\frac{\gamma}{\nu},\end{equation}
and estimate (\ref{3.281}) are valid.
\end{lemma}
\begin{proof}{\bf of lemma \ref{lem3.8}.}
Multiplying equation (\ref{3.27}) by $\mathcal{S}_{\mu}^{\nu}$ and integrating on $\torus^{2}$ give
\begin{equation}\label{3.38} \begin{array}{rr}
    \ds\mu\|\mathcal{S}^{\nu}_{\mu}(\theta,\cdot)\|_{2}^{2}+\frac{1}{2}\frac{d(\|\mathcal{S}^{\nu}_{\mu}(\theta,\cdot)\|_{2}^{2})}{d\theta}+\int_{\torus^{2}}
    (\widetilde{\mathcal{A}}_{\epsilon}+\nu)|\nabla\mathcal{S}^{\nu}_{\mu}(\theta,\cdot)|^{2}dx~~~~~\\
   \ds =    \int_{\torus^{2}}
    \nabla\cdot\widetilde{\mathcal{C}}_{\epsilon}\mathcal{S}^{\nu}_{\mu}(\theta,\cdot)\,dx
    =-\int_{\torus^{2}}\widetilde{\mathcal{C}}_{\epsilon}\nabla\mathcal{S}^{\nu}_{\mu}(\theta,\cdot)\,dx.\end{array}
    \end{equation}
Integrating (\ref{3.38}), from 0 to 1, with respect to $\theta$ gives
 \begin{equation}\label{3.39}\begin{array}{rr}
    \ds\mu\|\mathcal{S}^{\nu}_{\mu}\|^2_{L^{2}_{\#}(\mathbb{R}\times\torus^{2})}+
    \int_{0}^{1}\int_{\torus^{2}}(\widetilde{\mathcal{A}}_{\epsilon}+\nu)|\nabla\mathcal{S}^{\nu}_{\mu}|^{2}\,dxd\theta~~~~~~~~~~~\\
    \ds=-\int_{0}^{1}\int_{\torus^{2}}\widetilde{\mathcal{C}}_{\epsilon}\nabla\mathcal{S}^{\nu}_{\mu}(\theta,\cdot)\,dx
    \leq\gamma\|\nabla\mathcal{S}^{\nu}_{\mu}\|_{L^{2}_{\#}(\mathbb{R}\times\torus^{2})}.\end{array}
 \end{equation}
 Since $\widetilde{\mathcal{A}}_{\epsilon}+\nu\geq\nu,$ we deduce from (\ref{3.39})
 \begin{equation}\label{3.40}
   \nu \|\nabla\mathcal{S}^{\nu}_{\mu}\|^{2}_{L^{2}_{\#}(\mathbb{R}\times\torus^{2})}\leq\gamma\|\nabla\mathcal{S}^{\nu}_{\mu}\|_{L^{2}_{\#}(\mathbb{R}\times\torus^{2})},
 \end{equation}
 giving (\ref{3.371}).

 Multiplying (\ref{3.27}) by $\frac{\partial\mathcal{S}_{\mu}^{\nu}}{\partial\theta}$ and integrating on $\torus^{2}$ give
 \begin{equation}\label{3.41}
    \mu\frac{d\|\mathcal{S}_{\mu}^{\nu}\|_{2}^{2}}{d\theta}+
    \left\|\frac{\partial\mathcal{S}_{\mu}^{\nu}}{\partial\theta}\right\|_{2}^{2}+\int_{\torus^{2}}(\widetilde{\mathcal{A}}_{\epsilon}+\nu)
    \nabla\mathcal{S}^{\nu}_{\mu}\cdot\nabla\frac{\partial\mathcal{S}_{\mu}^{\nu}}{\partial\theta}\,dx=\int_{\torus^{2}}\nabla\cdot\widetilde{\mathcal{C}}_{\epsilon}\frac{\partial\mathcal{S}_{\mu}^{\nu}}{\partial\theta}\,dx
 \end{equation}
 Since
 \begin{equation}\label{3.42}
 \frac{\ds d\left(\int_{\torus^{2}}(\widetilde{\mathcal{A}}_{\epsilon}+\nu)|\nabla\mathcal{S}^{\nu}_{\mu}|^{2}dx\right)}{d\theta}=\\
 \int_{\torus^{2}}\frac{\partial\widetilde{\mathcal{A}}_{\epsilon}}{\partial\theta}|\nabla\mathcal{S}_{\mu}^{\nu}|^{2}dx+
 2\int_{\torus^{2}}
 (\widetilde{\mathcal{A}}_{\epsilon}+\nu)\nabla\mathcal{S}_{\mu}^{\nu}\cdot\nabla\frac{\partial\mathcal{S}_{\mu}^{\nu}}{\partial\theta}\,dx,
 \end{equation}
 integrating (\ref{3.41}) with respect to $\theta$ from 0 to 1
gives
\begin{equation}\label{3.43}\begin{array}{r}
    \ds\left\|\frac{\partial\mathcal{S}_{\mu}^{\nu}}{\partial\theta} \right\|^{2}_{L^{2}_{\#}(\mathbb{R}\times\torus^{2})}=\int_{0}^{1}\int_{\torus^{2}}\frac{\partial\widetilde{\mathcal{A}}_{\epsilon}}{\partial\theta}|\nabla\mathcal{S}_{\mu}^{\nu}|^{2}dxd\theta+\int_{0}^{1}\int_{\torus^{2}}\frac{\partial\widetilde{\mathcal{C}}_{\epsilon}}{\partial\theta}\nabla\mathcal{S}_{\mu}^{\nu}dxd\theta~~~~~~~~~\\
    \ds\leq\gamma(\|\nabla\mathcal{S}_{\mu}^{\nu}\|^{2}_{L^{2}_{\#}(\mathbb{R}\times\torus^{2})}+\|\nabla\mathcal{S}_{\mu}^{\nu}\|_{L^{2}_{\#}(\mathbb{R}\times\torus^{2})}).
\end{array}\end{equation}
Because of (\ref{3.371}), we can deduce that (\ref{3.281}) is also true, ending the proof of lemma \ref{lem3.8}
\end{proof}
\begin{lemma}\label{lem3.9}
For any $\mu>0,\,\,\nu>0$ and $\epsilon>0,$ the solution to (\ref{3.27}) satisfies
\begin{equation}\label{3.44}
    \|\Delta\mathcal{S}_{\mu}^{\nu}\|^{2}_{L^{2}_{\#}(\mathbb{R}\times\torus^{2})}\leq2\frac{\gamma^{2}}{\nu^{2}}(\frac{\gamma}{\nu}+1),
\end{equation}
and estimate (\ref{3.282}).
\end{lemma}
\begin{proof}{\bf of  lemma \ref{lem3.9}.} Multiplying (\ref{3.27}) by $-\Delta\mathcal{S}_{\mu}^{\nu}$ and integrating with respect to $x\in\torus^{2}$ yields,
\begin{equation}\label{3.45}
    \mu\|\nabla\mathcal{S}_{\mu}^{\nu}\|_{2}^{2}+\frac{1}{2}\frac{d\left(\|\nabla\mathcal{S}_{\mu}^{\nu}\|_{2}^{2}\right)}{d\theta}
    +\int_{\torus^{2}}\nabla\cdot((\widetilde{\mathcal{A}}_{\epsilon}+\nu)
    \nabla\mathcal{S}_{\mu}^{\nu})\,\Delta\mathcal{S}_{\mu}^{\nu}dx=
    -\int_{\torus^{2}}\nabla\cdot\widetilde{\mathcal{C}}_{\epsilon}\Delta\mathcal{S}_{\mu}^{\nu}dx,
\end{equation}
or
\begin{equation}\label{3.46}
    \ds\mu\|\nabla\mathcal{S}_{\mu}^{\nu}\|_{2}^{2}+\frac{1}{2}\frac{d\left(\|\nabla\mathcal{S}_{\mu}^{\nu}\|_{2}^{2}\right)}{d\theta}
    +\int_{\torus^{2}}(\widetilde{\mathcal{A}}_{\epsilon}+\nu)|\Delta\mathcal{S}_{\mu}^{\nu}|^{2}dx=
    -\ds\int_{\torus^{2}}\nabla\widetilde{\mathcal{A}}_{\epsilon}\cdot\nabla\mathcal{S}_{\mu}^{\nu}\Delta\mathcal{S}_{\mu}^{\nu}dx - \int_{\torus^{2}}\nabla\cdot\widetilde{\mathcal{C}}_{\epsilon}\,\Delta\mathcal{S}_{\mu}^{\nu}dx.
\end{equation}
Since for any real number $U$ and $V,$
\begin{equation}\label{3.47}
    |UV|\leq\frac{\widetilde{\mathcal{A}}_{\epsilon}+\nu}{4}U^{2}+\frac{1}{\widetilde{\mathcal{A}}_{\epsilon}+\nu}V^{2},
\end{equation}
using this formula with $U=\Delta\mathcal{S}_{\mu}^{\nu}$ and 
$V=\nabla\widetilde{\mathcal{A}}_{\epsilon}\cdot\nabla\mathcal{S}_{\mu}^{\nu}$ we get
\begin{equation}\label{3.48}
   \left |\int_{\torus^{2}}\nabla\widetilde{\mathcal{A}}_{\epsilon}\cdot\nabla\mathcal{S}_{\mu}^{\nu}\,\,\Delta\mathcal{S}_{\mu}^{\nu}dx \right|\leq
    \int_{\torus^{2}}\frac{\widetilde{\mathcal{A}}_{\epsilon}+\nu}{4}|\Delta\mathcal{S}_{\mu}^{\nu}|^{2}dx
    +\int_{\torus^{2}}\frac{|\nabla\widetilde{\mathcal{A}}_{\epsilon}|^{2}}{\widetilde{\mathcal{A}}_{\epsilon}+\nu}|\nabla\mathcal{S}_{\mu}^{\nu}|^{2}dx
\end{equation}
In the same way, applying (\ref{3.47}) with $U=\Delta\mathcal{S}_{\mu}^{\nu}$ and $V=\nabla\cdot\widetilde{\mathcal{C}}_{\epsilon}$ we get
\begin{equation}\label{3.49}
    \left|\int_{\torus^{2}}\nabla\cdot\widetilde{\mathcal{C}}_{\epsilon}\,\Delta\mathcal{S}_{\mu}^{\nu}dx\right|\leq
 \int_{\torus^{2}}\frac{\widetilde{\mathcal{A}}_{\epsilon}+\nu}{4}|\Delta\mathcal{S}_{\mu}^{\nu}|^{2}dx
    +\int_{\torus^{2}}\frac{|\nabla\widetilde{\mathcal{C}}_{\epsilon}|^{2}}{\widetilde{\mathcal{A}}_{\epsilon}+\nu}dx.
\end{equation}
Using (\ref{3.48}) and (\ref{3.49}) in (\ref{3.46}) gives
\begin{equation}\label{3.50}
    \mu\|\nabla\mathcal{S}_{\mu}^{\nu}\|_{2}^{2}+\frac{1}{2}\frac{d\left(\|\nabla\mathcal{S}_{\mu}^{\nu}\|_{2}^{2}\right)}{d\theta}
    +\int_{\torus^{2}}\frac{\widetilde{\mathcal{A}}_{\epsilon}+\nu}{2}|\Delta\mathcal{S}_{\mu}^{\nu}|^{2}dx\\
    \leq\frac{\gamma^{2}}{\nu}(\int_{\torus^{2}}|\nabla\mathcal{S}_{\mu}^{\nu}|^{2}dx+1),
\end{equation}
which, integrating in $\theta$ from 0 to 1, yields
\begin{equation}\label{3.51}
    \mu\|\nabla\mathcal{S}_{\mu}^{\nu}\|_{L^{2}_{\#}(\mathbb{R}\times\torus^{2})}^{2}+\int_{0}^{1}\int_{\torus^{2}}\frac{\widetilde{\mathcal{A}}_{\epsilon}+\nu}{2}|\Delta\mathcal{S}_{\mu}^{\nu}|^{2}dxdt\\
    \leq\frac{\gamma^{2}}{\nu}(\|\nabla\mathcal{S}_{\mu}^{\nu}\|^{2}_{L^{2}_{\#}(\mathbb{R}\times\torus^{2})}+1),.
\end{equation}
and in particular
\begin{equation}\label{3.53}
    \frac{\nu}{2}\|\Delta\mathcal{S}_{\mu}^{\nu}\|^2_{L^{2}_{\#}(\mathbb{R}\times\torus^{2})}\leq\frac{\gamma^{2}}{\nu}(\frac{\gamma}{\nu}+1),
\end{equation}
leading to (\ref{3.44}).\\
From lemma \ref{lem3.8}, we may deduce that there exists a $\theta_{0}\in[0,1]$ such that
\begin{equation}\label{3.53.111}
    \|\nabla\mathcal{S}_{\mu}^{\nu}(\theta_{0},\cdot)\|_{2}\leq\frac{\gamma}{\nu}.
\end{equation}
Beside this, from (\ref{3.50}) we can deduce
\begin{equation}\label{3.54}
    \frac{d\left(\|\nabla\mathcal{S}_{\mu}^{\nu}\|_{2}^{2}\right)}{d\theta}\leq\frac{2\gamma^{2}}{\nu}
    \left(\int_{\torus^{2}}|\nabla\mathcal{S}_{\mu}^{\nu}|^{2}+1\right),
\end{equation}
which gives, integrating from $\theta_{0}$ to any other $\theta_{1}\in[0,1],$
\begin{multline}\label{3.55}
    \|\nabla\mathcal{S}_{\mu}^{\nu}(\theta_{1},\cdot)\|_{2}^{2}-\|\nabla\mathcal{S}_{\mu}^{\nu}(\theta_{0},\cdot)\|_{2}^{2}
\leq\frac{2\gamma^{2}}{\nu}\int_{\theta_{0}}^{\theta_{1}}
    \left(\int_{\torus^{2}}|\nabla\mathcal{S}_{\mu}^{\nu}|^{2}dx+1\right)d\theta
\\
\leq \frac{2\gamma^{2}}{\nu}\left(\|\nabla\mathcal{S}_{\mu}^{\nu}\|^2_{L^{2}_{\#}(\mathbb{R}\times\torus^{2})} +1\right),~~
\end{multline}
giving the sought bound on $\|\nabla\mathcal{S}_{\mu}^{\nu}(\theta_{1},\cdot)\|_{2}^{2}$ for any $\theta_{1}$ or in other words (\ref{3.282}) (possibility changing the constant $\gamma_{3}.$)
 \end{proof}
 \begin{lemma}\label{lem3.10}
 For any $\mu>0,\nu>0$ and $\epsilon>0$ the solution to (\ref{3.27}) satisfies estimate (\ref{3.283}).
 \end{lemma}
\begin{proof}{{\bf of lemma \ref{lem3.10}}.}
At any $\theta\in\mathbb{R},$ we consider the Fourier expansion of $\mathcal{S}_{\mu}^{\nu}(\theta,\cdot)$:
\begin{equation}\label{3.56}
    \mathcal{S}_{\mu}^{\nu}(\theta,\cdot)=\sum_{n\in\mathbb{N}^{2}}\widehat{\mathcal{S}}_{n}e^{inx}.
\end{equation}
We have
\begin{equation}\label{3.57}
   \| \mathcal{S}_{\mu}^{\nu}(\theta,\cdot)\|_{2}^{2}=\sum_{n\in\mathbb{N}^{2}}|\widehat{\mathcal{S}}_{n}|^{2},
\end{equation}
\begin{equation}\label{3.58}
   \| \nabla\mathcal{S}_{\mu}^{\nu}(\theta,\cdot)\|_{2}^{2}=\sum_{n\in\mathbb{N}^{2}}|n|^{2}|\widehat{\mathcal{S}}_{n}|^{2},
\end{equation}
and, as a consequence of (\ref{3.280})
\begin{equation}\label{3.59}
    \widehat{\mathcal{S}}_{0}\,\,=\,\,0.
\end{equation}
Then, we obviously deduce from (\ref{3.57})-(\ref{3.59}) that
\begin{equation}\label{3.60}
    \| \mathcal{S}_{\mu}^{\nu}(\theta,\cdot)\|_{2}^{2}\leq \| \nabla\mathcal{S}_{\mu}^{\nu}(\theta,\cdot)\|_{2}^{2},
\end{equation}
for any $\theta,$ giving (\ref{3.280}) as a consequence. This ends the proof of lemma \ref{lem3.10}
\end{proof}

Finally, we can prove the following lemma
\begin{lemma}\label{lem3.11}
For any $\mu>0,\nu>0$ and $\epsilon>0$ the solution to (\ref{3.27}) satisfies estimate (\ref{3.29}).
\end{lemma}
\begin{proof}{\bf of lemma \ref{lem3.11}.}
Obviously, we get that $\frac{\partial\mathcal{S}_{\mu}^{\nu}}{\partial t}$ is a solution to
\begin{equation}\label{3.61}
    \mu\frac{\partial\mathcal{S}_{\mu}^{\nu}}{\partial t}+\frac{\ds\partial \left(\frac{\partial\mathcal{S}_{\mu}^{\nu}}{\partial t}\right)}{\partial\theta}-\nabla\cdot
    \left((\widetilde{\mathcal{A}}_{\epsilon}(t,\tau,\cdot,\cdot)+\nu)\,\nabla\frac{\partial\mathcal{S}_{\mu}^{\nu}}{\partial t}\right)=\nabla\cdot\check{\mathcal{C}}_{\epsilon},
\end{equation}
with
\begin{equation}\label{3.62}\begin{array}{c}
   \ds \check{\mathcal{C}}_{\epsilon}=\frac{\partial\widetilde{\mathcal{C}}_{\epsilon}}{\partial t}+\frac{\partial\widetilde{\mathcal{A}}_{\epsilon}}{\partial t}\nabla\mathcal{S}_{\mu}^{\nu},\\
   \ds \nabla\cdot\check{\mathcal{C}}_{\epsilon}=\frac{\partial\nabla\cdot\widetilde{\mathcal{C}}_{\epsilon}}{\partial t}+\frac{\partial(\nabla\widetilde{\mathcal{A}}_{\epsilon})}{\partial t}\cdot\nabla\mathcal{S}_{\mu}^{\nu}+\frac{\partial\widetilde{\mathcal{A}}_{\epsilon}}{\partial t}\Delta\mathcal{S}_{\mu}^{\nu}.\end{array}
\end{equation}
From the previous estimates, we can deduce that there exists a constant $\gamma_{5},$ which only depends on $\gamma$ and $\nu,$ such that
\begin{equation}\label{3.63}\left\{\begin{array}{ccc}
    \|\check{\mathcal{C}}_{\epsilon}\|_{L^{2}_{\#}(\mathbb{R}\times\torus^{2})}\leq\gamma_{5},\\
    \|\nabla\cdot\check{\mathcal{C}}_{\epsilon}\|_{L^{2}_{\#}(\mathbb{R} \times\torus^{2})}\leq\gamma_{5}.
    \end{array}\right.
\end{equation}
Proceeding in the same way as when proving lemma \ref{lem3.8}, we get
\begin{equation}\label{3.64}
    \int_{0}^{1}\int_{\torus^{2}}(\widetilde{\mathcal{A}}_{\epsilon}+\nu)\Big|\nabla\frac{\partial\mathcal{S}_{\mu}^{\nu}}{\partial t}\Big|^{2}dxd\theta\leq\gamma_{5}\Big\|\nabla\frac{\partial\mathcal{S}_{\mu}^{\nu}}{\partial t}\Big\|_{L^{2}_{\#}(\mathbb{R}\times\torus^{2})},
\end{equation}
giving
\begin{equation}\label{3.65}
    \Big\|\nabla\frac{\ds\partial\mathcal{S}_{\mu}^{\nu}}{\partial t}\Big\|^2_{L^{2}_{\#}(\mathbb{R}\times\torus^{2})}\leq\frac{\gamma_{5}}{\nu}\Big\|\nabla\frac{\partial\mathcal{S}_{\mu}^{\nu}}{\partial t}\Big\|_{L^{2}_{\#}(\mathbb{R} \times\torus^{2})}.
\end{equation}
Proceeding now as in the proof of lemma \ref{lem3.9}, we get
\begin{equation}\label{3.65.111}
    \frac{\ds d\Big\|\nabla\frac{\partial\mathcal{S}_{\mu}^{\nu}}{\partial t}\Big\|_{2}^{2}}{d\theta}\leq\frac{\gamma_{5}^{2}}{\nu}(\int_{\torus^{2}}\Big|\nabla\frac{\partial\mathcal{S}_{\mu}^{\nu}}{\partial t}\Big|^{2}+1),
\end{equation}
and using (\ref{3.64}), from which we can deduce
\begin{equation}\label{3.66}
    \Big\|\nabla\frac{\partial\mathcal{S}_{\mu}^{\nu}}{\partial t}(\theta_{0},\cdot)\Big\|_{2}\leq\frac{\gamma_{5}}{\nu},
\end{equation}
for a given $\theta_{0},$ we finally obtain
\begin{equation}\label{3.67}
    \Big\|\nabla\frac{\partial\mathcal{S}_{\mu}^{\nu}}{\partial t}\Big\|_{L^{\infty}_{\#}(\mathbb{R},L^{2}(\torus^{2}))}\,\,\textrm{is bounded}.
\end{equation}
Since the mean value of $\frac{\partial\mathcal{S}_{\mu}^{\nu}}{\partial t}(\theta,\cdot)$ is 0 for any $\theta\in\mathbb{R},$ using the same argument as in lemma \ref{lem3.10}, we conclude that the first estimate of (\ref{3.29}) is true. We can do the same for the second estimate, to end the proof of the lemma and also of  theorem \ref{th3.5}.
\end{proof}
\end{proof}

Since the estimates in theorem \ref{th3.5} do not depend on $\mu,$ making the process $\mu\longrightarrow0$ allows us to deduce the following theorem.
\begin{theorem}\label{th3.12}
Under the assumptions (\ref{3.12}),(\ref{3.13}) and (\ref{3.15}), for any $\nu>0,$ there exists a unique $\mathcal{S}^{\nu}=\mathcal{S}^{\nu}(t,\tau,\theta,x)\in L^{2}(\mathbb{R}\times\torus^{2}),$ periodic of period 1 with respect to $\theta$ solution to (\ref{3.26}) and submitted to the constraint
\begin{equation}\label{3.70}
    \sup_{\theta\in\mathbb{R}}\Big|\int_{\torus^{2}}\mathcal{S}^{\nu}(\theta,x)dx\Big|=0.
\end{equation}
Moreover, there exists a constant $\gamma_{3}$ which depends only on $\gamma$ and $\nu$ such that
\begin{equation}\label{3.71}\begin{array}{ccc}
    \ds\Big\|\frac{\partial\mathcal{S}^{\nu}}{\partial\theta}\Big\|_{L^{2}_{\#}(\mathbb{R}\times\torus^{2})}\leq\gamma_{3},\\
    \|\nabla\mathcal{S}^{\nu}\|_{L^{\infty}_{\#}(\mathbb{R},L^{2}(\torus^{2}))}\leq\gamma_{3},\qquad
   \ds \|\mathcal{S}^{\nu}\|_{L^{\infty}_{\#}(\mathbb{R},L^{2}(\torus^{2}))}\leq\gamma_{3},
\end{array}\end{equation}
\begin{equation}\label{3.72}
    \Big\|\frac{\partial\mathcal{S}^{\nu}}{\partial t}\Big\|_{L^{\infty}_{\#}(\mathbb{R},L^{2}(\torus^{2}))}\leq\gamma_{3},\qquad
\Big\|\frac{\partial\mathcal{S}^{\nu}}{\partial\tau}\Big\|_{L^{\infty}_{\#}(\mathbb{R},L^{2}(\torus^{2}))}\leq\sqrt{\epsilon}\,\gamma_{3}.
\end{equation}
\end{theorem}
\begin{proof}{\bf of theorem \ref{th3.12}.} As previously said, existence of
$\mathcal{S}^{\nu}$ follows from making $\mu$ tend to 0
 in (\ref{3.27}). Formulas (\ref{3.70})-(\ref{3.72}) directly come from estimates (\ref{3.280})-(\ref{3.283}).\\
Uniqueness is insured by (\ref{3.70}), once noticed that, if $\mathcal{S}^{\nu}$ and $\widetilde{\mathcal{S}}^{\nu}$ are two solutions of (\ref{3.26}), with constraint (\ref{3.70})$,\mathcal{S}^{\nu}-\widetilde{\mathcal{S}}^{\nu}$ is solution to
\begin{equation}\label{3.73}
\frac{\partial(\mathcal{S}^{\nu}-\widetilde{\mathcal{S}}^{\nu})}{\partial\theta}-\nabla\cdot((\widetilde{\mathcal{A}}_{\epsilon}+\nu)
\nabla(\mathcal{S}^{\nu}-\widetilde{\mathcal{S}}^{\nu}))=0,
    \end{equation}
from which we can deduce that
\begin{equation}\label{3.74}
    \nu\|\nabla(\mathcal{S}^{\nu}-\widetilde{\mathcal{S}}^{\nu})\|^{2}_{L^{2}_{\#}(\mathbb{R}\times\torus^{2})}=0,
\end{equation}
and  because of (\ref{3.70}), and its consequence:
\begin{equation}\label{3.75}
    \|\mathcal{S}^{\nu}-\widetilde{\mathcal{S}}^{\nu}\|_{L^{2}_{\#}(\mathbb{R},\torus^{2})}\leq
    \|\nabla(\mathcal{S}^{\nu}-\widetilde{\mathcal{S}}^{\nu})\|_{L^{2}_{\#}(\mathbb{R}\times\torus^{2})} ,
\end{equation}
that
\begin{equation}\label{3.76}
    \widetilde{\mathcal{S}}^{\nu}=\mathcal{S}^{\nu}.
\end{equation}
This ends the proof of theorem \ref{th3.12}.
\end{proof}

Having on hand
theorem \ref{th3.12}, we will set out the properties of $\mathcal{S^{\nu}}$ which will allow us to make the process $\nu\longrightarrow0.$
\begin{lemma}\label{lem3.13} Under assumptions 
(\ref{3.12}), (\ref{3.13}), (\ref{3.15}), (\ref{3.16}), (\ref{3.17}) and (\ref{3.171}) - (\ref{3.1710})
the solution $\mathcal{S}^{\nu}$ to (\ref{3.26}) given by theorem \ref{th3.12} satisfies
\begin{equation}\label{3.77}
    \Big(\int_{\theta_{\alpha}}^{\theta_{\omega}}\int_{\torus^{2}}|\nabla\mathcal{S}^{\nu}|^{2}dxd\theta\Big)^{1/2}\leq\frac{\gamma}{\widetilde G_{thr}}.
\end{equation}
\end{lemma}
\begin{proof}{\bf of lemma \ref{lem3.13}.}
We proceed in a way similar to the proof of lemma \ref{lem3.8}. Multiplying (\ref{3.26}) by $\mathcal{S}^{\nu}$ and integrating in $x$ over $\torus^{2}$ and in $\theta$ over $[0,1]$ gives
\begin{equation}\label{3.78}
\int_{0}^{1}\int_{\torus^{2}}(\widetilde{\mathcal{A}}_{\epsilon}+\nu)|\nabla\mathcal{S}^{\nu}|^{2}dxd\theta\leq\int_{0}^{1}\int_{\torus^{2}}|\widetilde{C}_{\epsilon}|\,|\nabla\mathcal{S}^{\nu}|dxd\theta
\leq\gamma\int_{0}^{1}\int_{\torus^{2}}\sqrt{\widetilde{\mathcal{A}}_{\epsilon}}\,|\nabla\mathcal{S}^{\nu}|
dxd\theta,
\end{equation}
the last inequality being obtained from (\ref{3.17}). From (\ref{3.78}) we deduce
\begin{equation}\label{3.781}
    \left\|\sqrt{\widetilde{\mathcal{A}}_{\epsilon}}\,|\nabla\mathcal{S}^{\nu}|\right\|_{L^{2}_{\#}(\mathbb{R},L^{2}(\torus^{2}))}\leq\gamma.
\end{equation}
On another hand since (\ref{3.16}) is assumed, we have
\begin{equation}\label{3.782}
\sqrt{\widetilde{G}_{thr}}\left(\int_{\theta_\alpha}^{\theta_\omega}\int_{\torus^{2}}|\nabla\mathcal{S}^{\nu}|^{2}dxd\theta\right)^{1/2}\leq
\left(\int_{\theta_\alpha}^{\theta_\omega}\int_{\torus^{2}}\widetilde{\mathcal{A}}_{\epsilon}|\nabla\mathcal{S}^{\nu}|^{2}dxd\theta\right)^{1/2}
\leq \left\|\sqrt{\widetilde{\mathcal{A}}_{\epsilon}}\,|\nabla\mathcal{S}^{\nu}|\right\|_{L^{2}_{\#}(\mathbb{R},L^{2}(\torus^{2}))}.
\end{equation}
Formula (\ref{3.781}) and (\ref{3.782}) give (\ref{3.77}).
\end{proof}\\
As a direct consequence of lemma \ref{lem3.13} we have the following corollary
\begin{corollary}\label{coro3.14} There exists $\theta_{0}\in[\theta_{\alpha},\theta_{\omega}]$ such  that, under assumptions (\ref{3.12}), (\ref{3.13}), (\ref{3.15}), (\ref{3.16}), (\ref{3.17}) and (\ref{3.171}) - (\ref{3.1710}), 
$\mathcal{S}^{\nu}$ satisfies
\begin{equation}\label{3.79}
    \|\nabla\mathcal{S}^{\nu}(\theta_{0},\cdot)\|_{2}\leq\frac{\gamma}{\sqrt{\widetilde{G}_{thr}}}.
\end{equation}
\end{corollary}
\begin{lemma}\label{lem3.15}
Under assumptions (\ref{3.12}), (\ref{3.13}),  (\ref{3.15}), (\ref{3.16}), (\ref{3.17}) and (\ref{3.171}) - (\ref{3.1710})
the solution $\mathcal{S}^{\nu}$ of (\ref{3.26}) given by theorem \ref{th3.12} satisfies
\begin{equation}\label{3.80}
    \|\mathcal{S}^{\nu}\|^{2}_{L^{\infty}_{\#}(\mathbb{R},L^{2}(\torus^{2}))}\leq\frac{\gamma}{\sqrt{\widetilde{G}_{thr}}}+2\gamma.
\end{equation}
\end{lemma}
\begin{proof}{\bf of lemma \ref {lem3.15}.} First, because of (\ref{3.70}), doing the
 same as in the proof of lemma \ref{lem3.10}, we get from (\ref{3.79}) that their exists a $\theta_0$ such that:
\begin{equation}\label{3.81}
    \|\mathcal{S}^{\nu}(\theta_{0},\cdot)\|_{2}\leq\frac{\gamma}{\sqrt{\widetilde{G}_{thr}}}.
\end{equation}
Secondly, in any $\theta$ and any $x$ where $\widetilde{\mathcal{C}}_{\epsilon}(\theta,x)\neq0$ and $\widetilde{\mathcal{A}}_{\epsilon} (\theta,x)\neq0,$ applying formula (\ref{3.47}) with $U=|\nabla\mathcal{S}^{\nu}|$ and $V=|\widetilde{\mathcal{C}}_{\epsilon}|$ we obtain
\begin{eqnarray}\label{3.82}{
    |\widetilde{\mathcal{C}}_{\epsilon}\cdot\nabla \mathcal{S}^{\nu}|\leq\frac{\widetilde{\mathcal{A}}_{\epsilon}+\nu}{4}|\nabla\mathcal{S}^{\nu}|^2+\frac{\ds{\ds{\widetilde{\mathcal{C}}_{\epsilon}}}^{2}}{\widetilde{\mathcal{A}}_{\epsilon}+\nu} {}}~~~~~~
    \nonumber\\
    \leq\frac{\widetilde{\mathcal{A}}_{\epsilon}+\nu}{4}|\nabla\mathcal{S}^{\nu}|^{2}+\gamma,
\end{eqnarray}
thanks to assumption (\ref{3.17}). Hence multiplying (\ref{3.26}) by $\mathcal{S}^{\nu}$ and integrating over $\torus^{2}$ yields
\begin{eqnarray}\label{3.83}{
\frac{1}{2}\frac{d\|\mathcal{S}^{\nu}(\theta,\cdot)\|_{2}^{2}}{d\theta}+
\int_{x\in\torus^{2},\,\,\widetilde{\mathcal{C}}_{\epsilon}(\theta,x)=0}(\widetilde{\mathcal{A}}_{\epsilon}
+\nu)|\nabla\mathcal{S}^{\nu}(\theta,\cdot)|^{2}dx {}}~~~~~~~~~~~~~~~~~~~~
\nonumber\\
+\int_{x\in\torus^{2},\,\,\widetilde{\mathcal{C}}_{\epsilon}(\theta,x)\neq0}(\widetilde{\mathcal{A}}_{\epsilon}+\nu)|\nabla\mathcal{S}^{\nu}
(\theta,\cdot)|^{2}dx~~~~~~~~~~~~~
\nonumber\\
\leq\int_{x\in\torus^{2},\,\,\widetilde{\mathcal{C}}_{\epsilon}(\theta,x)\neq0}|
\widetilde{\mathcal{C}}_{\epsilon}(\theta,x)\cdot\nabla\mathcal{S}^{\nu}(\theta,\cdot)|dx~~~~~~~~~~~~~
\nonumber\\~~~~~~~~
\leq\int_{x\in\torus^{2},\,\,\widetilde{\mathcal{C}}_{\epsilon}(\theta,x)\neq0}\frac{\widetilde{\mathcal{A}}_{\epsilon}+\nu}{4}|\nabla\mathcal{S}^{\nu} (\theta,\cdot) |^{2}dx+\int_{\torus^{2}}\gamma \,dx.\end{eqnarray}
Passing the first term of the right hand side in the left hand side yields
\begin{equation}\label{3.84}
    \frac{d\|\mathcal{S}^{\nu}(\theta,\cdot)\|_{2}^{2}}{d\theta}\,\,\leq2\gamma.
\end{equation}
Coupling (\ref{3.81}) and (\ref{3.84}) allows us to deduce
\begin{equation}\label{3.85}
    \|\mathcal{S}^{\nu}(\theta,\cdot)\|_{2}^{2}\,\,\leq\,\,\frac{\gamma}{\sqrt{\widetilde{G}_{thr}}}+2\gamma,
\end{equation}
for any $\theta\in[\theta_{0},\theta_{0}+1]$ and because of the periodicity of $\mathcal{S}^{\nu},$ inequality (\ref{3.80}).
\end{proof}

\begin{lemma}\label{lem3.16}
Under assumptions (\ref{3.12}), (\ref{3.13}),  (\ref{3.15}), (\ref{3.16}), (\ref{3.17}) and (\ref{3.171}) - (\ref{3.1710}) 
the solution $\mathcal{S}^{\nu}$ to (\ref{3.26}) given by theorem \ref{th3.12} satisfies
\begin{equation}\label{3.86}
    \left\|\frac{\partial\mathcal{S}^{\nu}}{\partial t}\right\|^{2}_{L^{\infty}_{\#}(\mathbb{R},L^{2}(\torus^{2}))}\leq\frac{\gamma+\gamma^{3}}{\sqrt{\widetilde G_{thr}}}+2\gamma,\quad
  \left\|\frac{\partial\mathcal{S}^{\nu}}{\partial \tau} \right\|^{2}_{L^{\infty}_{\#}(\mathbb{R},L^{2}(\torus^{2}))}\leq\epsilon\left(\frac{\gamma+\gamma^{3}}{\sqrt{\widetilde G_{thr}}}+2\gamma\right).
    \end{equation}
\end{lemma}
\begin{proof}{\bf of lemma \ref{lem3.16}.}
In a first step, remembering inequality (\ref{3.781}) proved in the beginning of the proof of lemma \ref{lem3.13}, and using (\ref{3.175}), we deduce
\begin{equation}\label{3.87}
    \left\|\frac{\partial(\nabla\widetilde{\mathcal{A}}_{\epsilon})}{\partial t}\,\,\nabla\mathcal{S}^{\nu}\right\|_{L^{2}_{\#}(\mathbb{R},L^{2}(\torus^{2}))}\,\,\leq \,\,\gamma\left\|\sqrt{\widetilde{\mathcal{A}}_{\epsilon}}\nabla\mathcal{S}^{\nu}\right\|_{L^{2}(\mathbb{R},L^{2}(\torus^{2}))}\leq\gamma^{2}.
\end{equation}

In a second step, following the way to prove lemma \ref{lem3.9}, we multiply (\ref{3.26}) by $-\Delta\mathcal{S}^{\nu}$ and we integrate the resulting equality in $x\in\torus^{2}$ to get
\begin{equation}\label{3.88}
    \frac{1}{2}\frac{d\|\nabla\mathcal{S}^{\nu}\|_{2}^{2}}{d\theta}+\int_{\torus^{2}}(\widetilde{\mathcal{A}}_{\epsilon}+\nu)|\Delta\mathcal{S}^{\nu}|^{2}dx=\,\,
    -\int_{\torus^{2}}\nabla\widetilde{\mathcal{A}}_{\epsilon}\cdot\nabla\mathcal{S}^{\nu}\,\,\Delta\mathcal{S}^{\nu}dx-\int_{\torus^{2}}\nabla\cdot\widetilde{\mathcal{C}}_{\epsilon}\,\,\Delta\mathcal{S}^{\nu}dx,
\end{equation}
which using (\ref{3.48}) and (\ref{3.49}) and integrating in $\theta$ over $[0,1]$ gives
\begin{eqnarray}\label{3.89}{    \int_{0}^{1}\int_{\torus^{2}}\frac{\widetilde{\mathcal{A}}_{\epsilon}+\nu}{2}|\Delta\mathcal{S}^{\nu}|^{2}dxd\theta\leq
   {} }
   {} \int_{0}^{1}\int_{\torus^{2}}\frac{|\nabla\widetilde{\mathcal{A}}_{\epsilon}|^{2}}{\widetilde{\mathcal{A}}_{\epsilon}+\nu}\,
    |\nabla\mathcal{S}^{\nu}|^{2}dxd\theta  +\int_{0}^{1}\int_{\torus^{2}}\frac{|\nabla\widetilde{\mathcal{A}}_{\epsilon}|^{2}}{\widetilde{\mathcal{A}}_{\epsilon}+\nu}dxd\theta,
\end{eqnarray}
Now using (\ref{3.173}), we deduce
\begin{equation}\label{3.90}
    \int_{0}^{1}\int_{\torus^{2}}\widetilde{\mathcal{A}}_{\epsilon}|\Delta\mathcal{S}^{\nu}|^{2}dxd\theta\leq
    \int_{0}^{1}\int_{\torus^{2}}\widetilde{\mathcal{A}}_{\epsilon}|\nabla\mathcal{S}^{\nu}|^{2}dxd\theta+\gamma\leq\gamma^{2}+\gamma,
\end{equation}
thanks to inequality (\ref{3.781}). Finally, because of (\ref{3.174}), we get from (\ref{3.90})
\begin{eqnarray}\label{3.91}{
\left\|\frac{\partial\widetilde{\mathcal{A}}_{\epsilon}}{\partial t}\,\Delta\mathcal{S}^{\nu}\right\|_{L^{2}_{\#}(\mathbb{R}\times\torus^{2})}\leq{} }
{} \gamma\left\|\sqrt{\widetilde{\mathcal{A}}_{\epsilon}}\Delta\mathcal{S}^{\nu}\right\|_{L^{2}(\mathbb{R}\times\torus^{2})}\leq\gamma+\gamma\sqrt{\gamma}.
\end{eqnarray}

In the third step, we set out the equation to which $\frac{\partial\mathcal{S}^{\nu}}{\partial t}$ 
is a solution.
 In a way similar to the one followed in the proof of lemma \ref{lem3.11}, we obtain
\begin{equation}\label{3.92}
    \frac{\ds\partial\left(\frac{\partial\mathcal{S}^{\nu}}{\partial t}\right)}{\partial\theta}-\nabla\cdot\left((\widetilde{\mathcal{A}}_{\epsilon}(t,\tau,\cdot,\cdot)+\nu)\,\nabla\frac{\partial\mathcal{S}^{\nu}}{\partial t}\right)=\nabla\cdot\check{\mathcal{C}}_{\epsilon}
\end{equation}
where
\begin{equation}\label{3.93}\begin{array}{ccc}
    \ds\check{\mathcal{C}}_{\epsilon}=\frac{\partial\widetilde{\mathcal{C}}_{\epsilon}}{\partial t}+\frac{\partial\widetilde{\mathcal{A}}_{\epsilon}}{\partial t}\cdot\nabla\mathcal{S}^{\nu},\vspace{3pt}\\
    \ds\nabla\cdot\check{\mathcal{C}}_{\epsilon}=\frac{\partial(\nabla\cdot\widetilde{\mathcal{C}}_{\epsilon})}{\partial t}+\frac{\partial\nabla\widetilde{\mathcal{A}}_{\epsilon}}{\partial t}\cdot\nabla\mathcal{S}^{\nu}+\frac{\partial\widetilde{\mathcal{A}}_{\epsilon}}{\partial t}\Delta\mathcal{S}^{\nu}.\end{array}
\end{equation}
Multiplying equation (\ref{3.92}) by $\frac{\partial\mathcal{S}^{\nu}}{\partial t}$ and integrating in $x\in\torus^{2},$ in the same spirit as in the proofs of lemma \ref{lem3.8} and \ref{lem3.11}, we deduce
\begin{eqnarray}\label{3.94}{
    \frac{1}{2}\frac{\ds\partial\Big\|\frac{\partial\mathcal{S}^{\nu}}{\partial t}\Big\|^{2}_{2}}{\partial \theta}+\int_{\torus^{2}}(\widetilde{\mathcal{A}}_{\epsilon}+\nu)\bigg|\nabla\frac{\partial\mathcal{S}^{\nu}}{\partial t}\bigg|^{2}dx\leq \int_{\torus^{2}}\bigg|\frac{\partial\widetilde{\mathcal{C}}_{\epsilon}}{\partial t}\bigg|\,\bigg|\nabla\frac{\partial\mathcal{S}^{\nu}}{\partial t}\bigg|dx {}}~~~~~~~~~~~~~
    \nonumber\\
    {}+\int_{\torus^{2}}\bigg|\frac{\partial\widetilde{\mathcal{A}}_{\epsilon}}{\partial t}\bigg|\,|\nabla\mathcal{S}^{\nu}|\,\bigg|\nabla\frac{\partial\mathcal{S}^{\nu}}{\partial t}\bigg|dx.
\end{eqnarray}
To estimate
 the right hand side of (\ref{3.94}), we first notice that, applying (\ref{3.1710}), we have
 \begin{eqnarray}\label{3.95}
 {    \int_{\torus^{2}}\bigg|\frac{\partial\widetilde{\mathcal{C}}_{\epsilon}}{\partial t}\bigg|\;\bigg|\nabla\frac{\partial\mathcal{S}^{\nu}}{\partial t}\bigg|dx\leq\gamma\int_{\torus^{2}}\sqrt{\widetilde{\mathcal{A}}_{\epsilon}}\bigg|\nabla\frac{\partial\mathcal{S}^{\nu}}{\partial t}\bigg|dx {} }
{}    \leq\gamma\bigg\|\sqrt{\widetilde{\mathcal{A}}_{\epsilon}}\;\left|\nabla\frac{\partial\mathcal{S}^{\nu}}{\partial t}\right|\bigg\|_{2}.
\end{eqnarray}
Then using (\ref{3.174}) we deduce
 \begin{eqnarray}\label{3.96}{
 \int_{\torus^{2}}\left|\frac{{\partial \widetilde{\mathcal{A}}}_{\epsilon}}{\partial t}\right| \,|\nabla\mathcal{S}^{\nu}| \left|\nabla\frac{\partial\mathcal{S}^{\nu}}{\partial t}\right|dx\leq
 \left\|\sqrt{\bigg|\frac{\partial\widetilde{\mathcal{A}}_{\epsilon}}{\partial t}}\bigg|\, |\nabla\mathcal{S}^{\nu}|\right\|_{2}\,\left\|\sqrt{\bigg|\frac{\partial\widetilde{\mathcal{A}}_{\epsilon}}{\partial t}}\bigg| \Big|\nabla\frac{\partial\mathcal{S}^{\nu}}{\partial t}\Big|\right\|_{2}{} }~~~~~~~~
 \nonumber\\
{}
\leq \gamma^{2}\left\|\sqrt{\widetilde{{\mathcal{A}}}_{\epsilon}} \left|\nabla\mathcal{S}^{\nu}\right|\right\|_{2}\,
\left\|\sqrt{\widetilde{{\mathcal{A}}}_{\epsilon}}\left| \nabla\frac{\partial\mathcal{S}^{\nu}}{\partial t}\right|\right\|_{2}.
\end{eqnarray}
As a consequence of (\ref{3.95}), (\ref{3.96}) and (\ref{3.781}), integrating (\ref{3.94}) in $\theta$ over $[0,1]$ yields
\begin{eqnarray}\label{3.97}{
\Bigg\|\sqrt{(\widetilde{{\mathcal{A}}}_{\epsilon}+\nu)}\,\bigg|\nabla\frac{\partial\mathcal{S}^{\nu}}{\partial t}\bigg|\Bigg\|_{L^{2}_
{\#}(\mathbb{R},L^{2}(\torus^{2}))}^{2}
\,\,\leq\,\,\gamma\Bigg\|\sqrt{\widetilde{{\mathcal{A}}}_{\epsilon}}
\bigg|\nabla\frac{\partial\mathcal{S}^{\nu}}{\partial t}\bigg|\Bigg\|_{L^{2}_{\#}(\mathbb{R},L^{2}(\torus^{2}))} {} }~~~~~~~~~~~~~~~
\nonumber\\
+\gamma^{3}\Bigg\|\sqrt{\widetilde{{\mathcal{A}}}_{\epsilon}}\bigg|\nabla\frac{\partial\mathcal{S}^{\nu}}{\partial t}\bigg |\Bigg\|_{L^{2}_{\#}(\mathbb{R},L^{2}(\torus^{2}))},
\end{eqnarray}
where once again (\ref{3.781}) is also used.
\\
From this last inequality, we deduce
\begin{equation}\label{3.98}
    \left\|\sqrt{\widetilde{{\mathcal{A}}}_{\epsilon}}\left|\nabla\frac{\partial\mathcal{S}^{\nu}}{\partial t}\right|\right\|_{L^{2}_{\#}(\mathbb{R},L^{2}(\torus^{2}))}\leq \gamma+\gamma^{3},
\end{equation}
and then
\begin{equation}\label{3.99}
    \int_{\theta_{\alpha}}^{\theta_{\omega}}\left\|\nabla\frac{\partial\mathcal{S}^{\nu}}{\partial t}\right\|_{2}d\theta\,\,\leq\,\,\frac{\gamma+\gamma^{3}}{\sqrt{\widetilde G_{thr}}}.
\end{equation}

The fourth step consists in deducing from (\ref{3.99}) that there exists a $\theta_{0}\in[\theta_{\alpha},\theta_{\omega}]$ such that
\begin{equation}\label{3.100}
    \left\|\nabla\frac{\partial\mathcal{S}^{\nu}}{\partial t}(\theta_{0},\cdot)\right\|_{2}\leq\frac{\gamma+\gamma^{3}}{\sqrt{\widetilde G_{thr}}},
\end{equation}
and, since the mean value of $\frac{\partial\mathcal{S}^{\nu}}{\partial t}(\theta_{0},\cdot)$
is zero,
\begin{equation}\label{3.101}
    \left\|\frac{\partial\mathcal{S}^{\nu}}{\partial t}(\theta_{0},\cdot)\right\|_{2}\,\,\leq\,\,\frac{\gamma+\gamma^{3}}{\sqrt{\widetilde G_{thr}}}.
\end{equation}

The fifth and last step, consists in going back to (\ref{3.94}). Applying formula (\ref{3.47}) with $U=|\frac{\partial\widetilde{\mathcal{C}}_{\epsilon}}{\partial t}|$ and $V=|\nabla\frac{\partial\mathcal{S}^{\nu}}{\partial t}|$ to treat the first term of the right hand side of (\ref{3.94}) and with  $U=|\frac{\partial\widetilde{\mathcal{A}}_{\epsilon}}{\partial t}||\nabla\mathcal{S}^{\nu}|$ and $V=|\nabla\frac{\partial\mathcal{S}^{\nu}}{\partial t}|$
to treat the second, we get
\begin{multline}\label{3.102}{\frac{1}{2}\frac{\ds\partial\bigg(\Big\|\frac{\ds\partial\mathcal{S}^{\nu}}{\partial t}\Big\|_{2}^{2}\bigg)}{\partial\theta}+\int_{\torus^{2}}\frac{\widetilde{\mathcal{A}}_{\epsilon}+\nu}{2}\left|\nabla\frac{\partial\mathcal{S}^{\nu}}{\partial t}\right|^{2}dx{}}
\\
{}\leq\,\,\int_{\torus^{2}}\frac{\ds\Big|\frac{\partial\widetilde{\mathcal{C}}_{\epsilon}}{\partial t}\Big|^2}{\widetilde{\mathcal{A}}_{\epsilon}+\nu}dx+\int_{\torus^{2}}\frac{\ds\Big|\frac{\partial\widetilde{\mathcal{A}}_{\epsilon}}{\partial t}\Big|^{2}|\nabla\mathcal{S}^{\nu}|^{2}}{\widetilde{\mathcal{A}}_{\epsilon}+\nu}dx
\leq \gamma+\int_{\torus^{2}}\widetilde{\mathcal{A}}_{\epsilon}|\nabla\mathcal{S}^{\nu}|^{2}dx,
\end{multline}
where we used (\ref{3.1710}) and (\ref{3.174}) to find the last inequality. Integrating this last formula in $\theta$ over $[\theta_{0},\sigma]$ for any $\sigma>\theta_{0},$ we obtain, always remembering (\ref{3.781}),
\begin{equation}\label{3.103}
    \left\|\frac{\partial\mathcal{S}^{\nu}}{\partial t}(\sigma,\cdot)\right\|_{2}^{2}\leq \left\|\frac{\partial\mathcal{S}^{\nu}}{\partial t}(\theta_{0},\cdot)\right\|_{2}^{2}+2\gamma\leq\frac{\gamma+\gamma^{3}}{\sqrt{\widetilde G_{thr}}}+2\gamma.
\end{equation}
Inequality (\ref{3.103}) yields directly the first inequality of (\ref{3.86}), using the periodicity of $\mathcal{S}^{\nu}.$ The proof of the second inequality of (\ref{3.86}) is done in a similar way.
This ends the proof of lemma \ref{lem3.16}.
\end{proof}

~

As neither estimate (\ref{3.80}) nor estimate (\ref{3.86}) depend on $\nu,$ we can deduce that, extracting a subsequence, as $\nu\longrightarrow0,$ $\mathcal{S}^{\nu}\longrightarrow\mathcal{S}$ in $L^{\infty}_{\#}(\mathbb{R},L^{2}(\torus^{2})) \,\,\textrm{weak-*},$ and that the limit satisfies estimates looking like
 (\ref{3.80}) and (\ref{3.86}), but also a property of the type (\ref{3.70}), and that it is solution to equation (\ref{3.26}) with $\nu=0.$
In other words, we deduce the following theorem.
\begin{theorem}\label{th3.17}
Under assumptions (\ref{3.12}), (\ref{3.13}), (\ref{3.16}), (\ref{3.17}) and (\ref{3.171}) - (\ref{3.1710}) 
there exists a unique function $\mathcal{S}=\mathcal{S}(t,\tau,\theta,x)\in L_{\#}^{\infty}(\mathbb{R},L^{2}(\torus^{2}))$, periodic of period 1 with respect to $\theta,$ solution to
\begin{equation}\label{3.104}
    \frac{\partial\mathcal{S}}{\partial \theta}-\nabla\cdot(\widetilde{\mathcal{A}}_{\epsilon}(t,\tau,\cdot,\cdot)\nabla\mathcal{S})
    =\nabla\cdot\widetilde{\mathcal{C}}_{\epsilon}(t,\tau,\cdot,\cdot),
\end{equation}
and satisfying, for any $t,\tau,\theta\in\mathbb{R}^{+}\times\mathbb{R}\times\mathbb{R}$
\begin{equation}\label{3.104.1}
\int_{\torus^{2}}\mathcal{S}(t,\tau,\theta,x)dx=0.
\end{equation}
Moreover it satsfies: 
\begin{equation}\label{3.105}
    \|\mathcal{S}\|_{L^{\infty}_{\#}(\mathbb{R},L^{2}(\torus^{2}))}\leq\frac{\gamma}{\sqrt{\widetilde{G}_{thr}}}+2\gamma,
\end{equation}
\begin{equation}\label{3.106}
    \|\frac{\partial\mathcal{S}}{\partial t}\|^{2}_{L^{\infty}_{\#}(\mathbb{R},L^{2}(\torus^{2}))}\leq\frac{\gamma+\gamma^{3}}{\sqrt{\widetilde G_{thr}}}+2\gamma,
\end{equation}
\begin{equation}\label{3.107}
     \|\frac{\partial\mathcal{S}}{\partial \tau}\|^{2}_{L^{\infty}_{\#}(\mathbb{R},L^{2}(\torus^{2}))}\leq\epsilon \bigg(\frac{\gamma+\gamma^{3}}{\sqrt{\widetilde G_{thr}}}+2\gamma\bigg).
\end{equation}
\end{theorem}

\begin{remark}
Uniqueness of $\mathcal{S}$ is not gotten via the above evoked process
${\nu}\longrightarrow 0$, but directly comes from (\ref {3.104}).
Assuming that there are two solutions ${\mathcal{S}_1}$ and ${\mathcal{S}_2}$ to (\ref {3.104}),
we easily deduce that
\begin{equation}\label{3.106.1}
    \frac{d\left(\left\|{\mathcal{S}_1}- {\mathcal{S}_2}\right\|^2_2\right)}{d\theta}
    -\int_{\torus^{2}}\widetilde{\mathcal{A}}_{\epsilon}\left|\nabla({\mathcal{S}_1}- {\mathcal{S}_2}) \right|^2 dx   =0,
\end{equation}
which gives, because of the non-negativity of $\widetilde{\mathcal{A}}_{\epsilon}$,
\begin{equation}\label{3.106.2}
    \frac{d\left(\left\|{\mathcal{S}_1}- {\mathcal{S}_2}\right\|^2_2\right)}{d\theta}
    \leq 0.
\end{equation}
From (\ref {3.106.1}) we deduce that either
\begin{equation}\label{3.106.3}
  \widetilde{\mathcal{A}}_{\epsilon}\left|\nabla({\mathcal{S}_1}- {\mathcal{S}_2}) \right|^2 \equiv0,
\end{equation}
or, for any $\theta\in \R$,
\begin{equation}\label{3.106.4}
\left\|{\mathcal{S}_1}(\theta+1,\cdot)- {\mathcal{S}_2} (\theta+1,\cdot)\right\|^2_2
< \left\|{\mathcal{S}_1}(\theta,\cdot)- {\mathcal{S}_2} (\theta,\cdot)\right\|^2_2.
\end{equation}
As (\ref {3.106.4}) is not possible because of the periodicity of ${\mathcal{S}_1}$ and ${\mathcal{S}_2}$,
we deduce that (\ref {3.106.3}) is true.  Using this last information, we deduce, for instance
\begin{equation}\label{3.106.5}
\nabla({\mathcal{S}_1}- {\mathcal{S}_2})(\theta_\omega,\cdot) \equiv 0,
\end{equation}
yielding, because of property (\ref {3.104.1})
\begin{equation}\label{3.106.6}
\left\|{(\mathcal{S}_1}- {\mathcal{S}_2}) (\theta_\omega,\cdot)\right\|^2_2
\leq\left\|\nabla({\mathcal{S}_1}- {\mathcal{S}_2})(\theta_\omega,\cdot)\right\|^2_2.
\end{equation}
Injecting (\ref {3.106.3}) in (\ref {3.106.1}) yields
\begin{equation}\label{3.106.7}
    \frac{d\left(\left\|{\mathcal{S}_1}- {\mathcal{S}_2}\right\|^2_2\right)}{d\theta}
    = 0,
\end{equation}
and then
\begin{equation}\label{3.106.6.111}
\left\|{(\mathcal{S}_1}- {\mathcal{S}_2}) (\theta,\cdot)\right\|^2_2=0,
\end{equation}
for any $\theta\geq\theta_\omega$ and consequently or any $\theta\in\R$.
{\ \hfill \rule{0.5em}{0.5em}}
\end{remark}

Theorem \ref{th1} is the consequence of theorem \ref{th3.17}. We will now end its proof.
\\
{\bf End of the proof of theorem \ref{th1}.} Since, with definitions (\ref{3.1})-(\ref{3.4}) or (\ref{3.1}), (\ref{3.3}), (\ref{3.6}), (\ref{3.7}), properties (\ref{3.12}), (\ref{3.13}), (\ref{3.16}), (\ref{3.17}) and (\ref{3.171}) - (\ref{3.1710}) 
are consequences of assumptions (\ref{eq2}), (\ref{eq3}) and (\ref{eq4}) or (\ref{eq2}), (\ref{eq6}), (\ref{eq7}) and (\ref{eq8}), theorem \ref{th1} is a direct consequence of the following theorem.
\begin{theorem}\label{th3.18}
Under properties (\ref{3.12}), (\ref{3.13}),(\ref{3.15}), (\ref{3.16}), (\ref{3.17}) and (\ref{3.171}) - (\ref{3.1710}), 
for any $T,$ not depending on $\epsilon,$ equation (\ref{3.5}), with coefficients given by (\ref{3.1}) (coupled with (\ref{3.2}) or (\ref{3.6})) and (\ref{3.3}) (coupled with (\ref{3.4}) or (\ref{3.7})), has a unique solution $z^{\epsilon}\in L^{\infty}([0,T];L^{2}(\torus^{2})).$ This solution satisfies:
\begin{equation}\label{3.1080}
    \|z^{\epsilon}\|_{L^{\infty}([0,T],L^{2}(\torus^{2}))}\,\,\leq\widetilde{\gamma}
\end{equation}
where $\widetilde{\gamma}$ is a constant which does not depend on $\epsilon.$
\end{theorem}
\begin{proof}{\bf of theorem \ref {th3.18}.} To prove uniqueness,
 we consider $z_{1}^{\epsilon}$ and $z_{2}^{\epsilon}$ two solutions of
  (\ref{3.5}). Their difference is then solution to
\begin{equation}\label{3.108}\left\{\begin{array}{ccc}
    \ds \frac{\partial(z_{1}^{\epsilon}-z_{2}^{\epsilon})}{\partial t} -\frac{1}{\epsilon}\nabla\cdot\left(\widetilde{\mathcal{A}}_{\epsilon}\nabla(z_{1}^{\epsilon}-z_{2}^{\epsilon})\right)=0,\,\,
    \\
   \left( z_{1}^{\epsilon}-z_{2}^{\epsilon}\right)_{|t=0}=0,
\end{array}\right.\end{equation}
and multiplying the first equation of (\ref{3.108}) by $(z_{1}^{\epsilon}-z_{2}^{\epsilon})$ and integrating with respect to $x$ gives
\begin{equation}\label{3.109}
    \frac{d\left(\|z_{1}^{\epsilon}-z_{2}^{\epsilon}\|_{2}^{2}\right)}{dt}\,\,\leq\,\,0,
\end{equation}
yielding
\begin{equation}\label{3.110}
    \|z_{1}^{\epsilon}-z_{2}^{\epsilon}\|=0,\quad\textrm{for any}\,\, t,
\end{equation}
and giving uniqueness.

Concerning existence, corollary \ref{corEstLoc}  gives existence of $z^{\epsilon}$ on a time interval of length of some $\epsilon.$\\
Now we consider the function $\zep = \zep (t,x)=\mathcal{S}(t,\frac{t}{\sqrt{\epsilon}},\frac{t}{\epsilon},x)$ where $\mathcal{S}$ is given by theorem \ref{th3.17}. Since
\begin{equation}\label{3.111}
    \frac{\partial \zep}{\partial t}\,\,=\,\,\frac{\partial\mathcal{S}}{\partial t}(t,\frac{t}{\sqrt{\epsilon}},\frac{t}{\epsilon},x)+\frac{1}{\sqrt{\epsilon}}\frac{\partial\mathcal{S}}{\partial \tau}(t,\frac{t}{\sqrt{\epsilon}},\frac{t}{\epsilon},x)+\frac{1}{\epsilon}\frac{\partial\mathcal{S}}{\partial \theta}(t,\frac{t}{\sqrt{\epsilon}},\frac{t}{\epsilon},x),
\end{equation}
it is obvious to deduce from (\ref{3.104}) that $\zep$ is solution to
\begin{equation}\label{3.112}
    \ds \frac{\partial \zep}{\partial t}-\frac{1}{\epsilon}\nabla\cdot(\widetilde{\mathcal{A}}_{\epsilon}\nabla \zep)=
    \frac{1}{\epsilon}\nabla\cdot\mathcal{C}^{\epsilon}\\
    +\frac{\partial\mathcal{S}}{\partial t}(t,\frac{t}{\sqrt{\epsilon}},\frac{t}{\epsilon},x)+\frac{1}{\sqrt{\epsilon}}\frac{\partial\mathcal{S}}{\partial \tau}(t,\frac{t}{\sqrt{\epsilon}},\frac{t}{\epsilon},x).
\end{equation}
From (\ref{3.112}) and (\ref{3.5}) we deduce that $z^{\epsilon}-\zep$
is solution  to
\begin{equation}\label{3.113}\left\{\begin{array}{ccc}
\ds\frac{\partial (z^{\epsilon}-\zep)}{\partial t}-\frac{1}{\epsilon}\nabla\cdot(\widetilde{\mathcal{A}}_{\epsilon}\nabla(z^{\epsilon}-\zep))\,\,=\,\,\frac{\partial\mathcal{S}}{\partial t}(t,\frac{t}{\sqrt{\epsilon}},\frac{t}{\epsilon},x)+\frac{1}{\sqrt{\epsilon}}\frac{\partial\mathcal{S}}{\partial \tau}(t,\frac{t}{\sqrt{\epsilon}},\frac{t}{\epsilon},x),\\
\ds(z^{\epsilon}-\zep)_{|t=0}=z_{0}-\mathcal{S}(0,0,\cdot).
\end{array}\right.\end{equation}
Multiplying (\ref{3.113}) by  $z^{\epsilon}-\zep$ and integrating in $x,$ using estimates (\ref{3.106}) and (\ref{3.107}) gives:
\begin{equation}\label{3.114}
    \frac{d\left(\|z^{\epsilon}-\zep\|_{2}^{2}\right)}{dt}\,\,\leq2\sqrt{\frac{\gamma+\gamma^{3}}{\sqrt{\widetilde G_{thr}}}+2\gamma}\; \|z^{\epsilon}-\zep\|_{2},
\end{equation}
which gives
\begin{equation}\label{3.115}
    \|z^{\epsilon}(t,\cdot)-\zep(t,\cdot)\|_{2}\leq2\|z_{0}-\mathcal{S}(0,0,\cdot)\|_{2}\;\sqrt{\frac{\gamma+\gamma^{3}}{\sqrt{\widetilde  G_{thr}}}+2\gamma}\;t,
\end{equation}
which gives (\ref{3.1080}) with
\begin{equation}\label{3.116}
    \widetilde{\gamma}=2\|z_{0}-\mathcal{S}(0,0,\cdot)\|_{2}\sqrt{\frac{\gamma+\gamma^{3}}{\sqrt{\widetilde G_{thr}}}+2\gamma}\;T.
\end{equation}
Coupling local existence and estimate (\ref{3.1080}) yield global existence and then theorem \ref{th3.18} is true.
\end{proof}
As a consequence theorem \ref{th1} is also true.{\ \hfill \rule{0.5em}{0.5em}\\}

As a by-product of theorem \ref{th3.12}, using a way similar of the one used to prove theorem \ref{th3.18} we can obtain a theorem giving long-term existence of space-periodic solution to parabolic equation.
\begin{theorem}\label{th3.19}
Under assumptions (\ref{3.12}), (\ref{3.13}) and (\ref{3.15}), for any $\nu>0,$ any $\epsilon>0,$ and any $T$ not depending on $\epsilon,$ equation (\ref{3.20}) has a unique solution solution $z^{\epsilon,\nu}\in L^{\infty}([0,T], L^{2}(\torus^{2})).$ Moreover
$$\|z^{\epsilon,\nu}\|_{L^{\infty}([0,T], L^{2}(\torus^{2}))}\,\,\textrm{is bounded.}$$
\end{theorem}
\section{Homogenization, proof of theorem \ref{thAsyBeh}}\label{secHom}
We consider  equation (\ref{3.5}) 
where $\mathcal{A}^{\epsilon}$ and $\mathcal{C}^{\epsilon}$ are defined by formulas (\ref{3.1}) coupled with (\ref{3.6}) and (\ref{3.3}) coupled with (\ref{3.7}).
Our aim consists in deducing the equations satisfied by the limit of $z^{\epsilon}$ solution to (\ref{3.5}) as $\epsilon\longrightarrow0.$
\\~\\
It is obvious that 
\begin{eqnarray}\label{H2}
{\mathcal{A}^{\epsilon}(t,x)\,\,\textrm{two scale converges to}\,\,  \widetilde{\mathcal{A}}(t,\theta,x)\in L^{\infty}([0,T],L^{\infty}_\#(\R,L^2(\torus^{2}))) {} }~~~~~~~~
    \nonumber\\
\textrm{and}\,\, \mathcal{C}^{\epsilon}(t,x)\,\,\textrm{two scale converges to}\,\, \widetilde{\mathcal{C}}(t,\theta,x),
\end{eqnarray}
with 
\begin{equation}\label{EqHomcof}
\widetilde{\mathcal{A}}(t,\theta,x)=a\,g_a(|\mathcal{U}(t,\theta,x)|) \,\, \textrm{and}\,\,\widetilde{\mathcal{C}}(t,\theta,x)=c\,g_c(|\mathcal{U}(t,\theta,x)|)\,\frac{\mathcal{U}(t,\theta,x)}{|\mathcal{U}(t,\theta,x)|},
\end{equation}
and we have the following theorem.
\begin{theorem}\label{thHomSecHom}
Under assumptions (\ref{3.12}), (\ref{3.15}), (\ref{3.16}), (\ref{3.17}) and (\ref{3.171}) - (\ref{3.1710}), 
for any $T,$ not depending on $\epsilon,$  the sequence $(z^{\epsilon})$ of solutions to (\ref{3.5}), with coefficients given by (\ref{3.1}) coupled with (\ref{3.6}) and (\ref{3.3}) coupled with (3.7), two-scale converges to the profile $U\in L^{\infty}([0,T],L^{\infty}_\#(\R,L^2(\torus^{2})))$ solution to
\begin{equation}
\label{ee179}
\frac{\partial U}{\partial\theta}
-\nabla\cdot(\widetilde{\mathcal{A}}\nabla U)=\nabla \cdot\widetilde{\mathcal{C}},
\end{equation} where $\widetilde{\mathcal{A}}$ and $\widetilde{\mathcal{C}}$ are given by (\ref{EqHomcof}).
\end{theorem}
\begin{proof}{\bf of theorem \ref{thAsyBeh}.} Theorem \ref{thAsyBeh} is a direct consequence of
theorem \ref{thHomSecHom}.
\end{proof}
~\\
\begin{proof}{\bf of theorem \ref{thHomSecHom}.}
Defining test function $\psi^{\epsilon}(t,x)=\psi(t,\frac{t}{\epsilon},x)$ for any $\psi(t,\theta,x)$,
regular with compact support in $[0,T)\times\torus^{2}$ and periodic in $\theta$ with period 1,
multiplying (\ref{3.5}) by $\psi^{\epsilon}$  and integrating in $[0,T)\times\torus^{2}$ gives
\begin{equation}\label{H3}\int_{\torus^{2}}\int_{0}^{T}\frac{\partial z^{\epsilon}}{\partial t}\psi^{\epsilon}dtdx-
    \frac{1}{\epsilon}\int_{\torus^{2}}\int_{0}^{T}\nabla\cdot(\mathcal{A}^{\epsilon}\nabla z^{\epsilon})\psi^{\epsilon}dtdx=\frac{1}{\epsilon}\int_{\torus^{2}}\int_{0}^{T}\nabla\cdot\mathcal{C}^{\epsilon}\psi^{\epsilon}dtdx.
\end{equation}
Then integrating by parts in the first integral over $[0,T)$ and using the Green formula in $\torus^{2}$ in the second integral  we have
\begin{eqnarray}\label{H4}
    \lefteqn{-\int_{\torus^{2}}z_{0}(x)\psi(0,0,x)dx-\int_{\torus^{2}}\int_{0}^{T}\frac{\partial\psi^{\epsilon}}{\partial t}z^{\epsilon}dtdx {} }
                       \nonumber\\
    & & {}+\frac{1}{\epsilon}\int_{\torus^{2}}\int_{0}^{T}\mathcal{A}^{\epsilon}\nabla z^{\epsilon}\nabla\psi^{\epsilon}dtdx=\frac{1}{\epsilon}\int_{\torus^{2}}\int_{0}^{T}
    \nabla\cdot\mathcal{C}^{\epsilon}\psi^{\epsilon}dtdx.
\end{eqnarray}
Again using the green formula in the third integral we obtain
\begin{eqnarray}\label{H4.111}
    \lefteqn{-\int_{\torus^{2}}z_{0}(x)\psi(0,0,x)\,dx-\int_{\torus^{2}}\int_{0}^{T}\frac{\partial\psi^{\epsilon}}{\partial t}z^{\epsilon}dtdx {} }
                       \nonumber\\
    & & {}-\frac{1}{\epsilon}\int_{\torus^{2}}\int_{0}^{T}z^{\epsilon}\,\nabla\cdot(\mathcal{A}^{\epsilon}\nabla\psi^{\epsilon})\,dtdx=
    \frac{1}{\epsilon}\int_{\torus^{2}}\int_{0}^{T}
    \nabla\cdot\mathcal{C}^{\epsilon}\psi^{\epsilon}dtdx.
\end{eqnarray}
But
\begin{equation}\label{H5}
 \frac{\partial \psi^{\epsilon}}{\partial t}
 =\left(\frac{\partial \psi}{\partial t}\right)^{\epsilon}
 +\frac{1}{\epsilon}\left(\frac{\partial\psi}{\partial \theta}\right)^{\epsilon},
\end{equation}
where
\begin{gather}\label{H4a}
   \left(\frac{\partial\psi}{\partial t}\right)^{\epsilon}(t,x)
   =\frac{\partial\psi}{\partial t}(t,\frac{t}{\epsilon},x)\,\,\textrm{ and }
     \left(\frac{\partial\psi}{\partial\theta}\right)^{\epsilon}(t,x)=\frac{\partial\psi}{\partial\theta}(t,\frac{t}{\epsilon},x),
\end{gather}
then we have
\begin{eqnarray}\label{H6}
    \lefteqn{ \int_{\torus^{2}}\int_{0}^{T}z^{\epsilon}\left( \left(\frac{\partial \psi}{\partial t}\right)^{\epsilon}+\frac{1}{\epsilon}\left(\frac{\partial\psi}{\partial \theta}\right)^{\epsilon}+\frac{1}{\epsilon}\nabla\cdot(\mathcal{A}^{\epsilon}\nabla\psi^{\epsilon})\right)dxdt {}}~~~~~~~~
    \nonumber\\
    & &+
\frac{1}{\epsilon}\int_{\torus^{2}}\int_{0}^{T}
    \nabla\cdot\mathcal{C}^{\epsilon}\psi^{\epsilon}dtdx=-\int_{\torus^{2}}z_{0}(x)\psi(0,0,x)\,dx.
    \end{eqnarray}
\\
Using the two-scale convergence due to Nguetseng \cite{nguetseng:1989} and 
Allaire \cite{allaire:1992} (see also Fr\'enod Raviart and Sonnendr\"{u}cker \cite{FRS:1999}),  if a sequence $f^{\epsilon}$ is bounded in $L^{\infty}(0,T,L^{2}(\torus^{2}))$, then there exists a profile 
$U(t,\theta,x)$, periodic of period 1 with respect to $\theta$, such that for all 
$\psi(t,\theta,x),$ regular with compact support with respect to $(t,x)$ and  periodic of period 1 
with respect to $\theta$, we have
 \begin{equation}\label{H7}
    \int_{\torus^{2}}\int_{0}^{T}f^{\epsilon}\psi^{\epsilon}dtdx
    \longrightarrow\int_{\torus^{2}}\int_{0}^{T}\int_{0}^{1}U\psi \,d\theta dtdx,
 \end{equation}
 for a subsequence extracted from $(f^{\epsilon})$.
 \\
Multiplying (\ref{H6}) by $\epsilon$ and passing to the limit as $\epsilon\rightarrow 0$ and using (\ref{H7})
we have
    \begin{eqnarray}\label{H8}{
\int_{\torus^{2}}\int_{0}^{T}\int_{0}^{1}U\frac{\partial \psi}{\partial\theta}\,d\theta dtdx+\lim_{\epsilon\rightarrow0}\int_{\torus^{2}}\int_{0}^{T}z^{\epsilon}\nabla\cdot(\mathcal{A}^{\epsilon}\nabla\psi^{\epsilon})\,dtdx {} }~~~~~~~~
\nonumber\\
{} =\lim_{\epsilon\rightarrow 0}\int_{\torus^{2}}\int_{0}^{T}\mathcal{C}^{\epsilon}\cdot\nabla\psi^{\epsilon} dtdx,
    \end{eqnarray}
    for an extracted subsequence.
As $\mathcal{A}^{\epsilon}$ and $\mathcal{C}^{\epsilon}$ are bounded (see (\ref{3.12}))  and $\psi^{\epsilon}$ is a regular function, $\mathcal{A}^{\epsilon}\nabla\psi^{\epsilon}$ and $\nabla\psi^{\epsilon}$ can be considered as test functions.
 Using (\ref{H2}) we have
\begin{equation}
   \int_{\torus^{2}}\int_{0}^{T}z^{\epsilon}\,\nabla\cdot(\mathcal{A}^{\epsilon}\nabla\psi^{\epsilon})dtdx\longrightarrow
    \int_{\torus^{2}}\int_{0}^{T}\int_{0}^{1}U\nabla\cdot(\widetilde{\mathcal{A}}\nabla\psi)\,d\theta dtdx,
\end{equation}
and
\begin{equation}\label{H9}
\int_{\torus^{2}}\int_{0}^{T}\mathcal{C}^{\epsilon}\cdot\nabla\psi^{\epsilon}dtdx 
\,\,\textrm{two scale converges to} \,\,\int_{\torus^{2}}\int_{0}^{T}\int_{0}^{1}\widetilde{\mathcal{C}}\cdot\nabla\psi \,d\theta dtdx.
\end{equation}
From this we obtain from (\ref{H8}) the equation satisfied by $U$:
\begin{equation}\label{H10}
    \frac{\partial U}{\partial \theta}-\nabla\cdot(\widetilde{\mathcal{A}}\nabla U)\,\,= \nabla\cdot\widetilde{\mathcal{C}},
    \end{equation}
which is equation (\ref{ee179}).

~

Existence and uniqueness of equation (\ref{ee179}) is given by theorem \ref{th3.17} (applied
with $\epsilon$-independent coefficient and right hand side).
\\
From this uniqueness, we can deduce that the whole sequence $(z^\epsilon)$ converges.

~

 Let us  characterize the homogenized equation for $\widetilde{\mathcal{A}}$ and $\widetilde{\mathcal{C}}.$
We recall that $\mathcal{A}_{\epsilon}$ and $\mathcal{C}_{\epsilon}$ are  given by formulas (\ref{3.2}) and (\ref{3.4}). Multiplying this equation by a test function $\psi_{\epsilon}$  using the two-scale limits
\begin{eqnarray}{\int_{\torus^{2}}\int_{0}^{T}\widetilde{\mathcal{A}}_{\epsilon}\psi^{\epsilon}dtdx=\int_{\torus^{2}}\int_{0}^{T}a
(1-b\epsilon\mathcal{M}(t,\theta,x))
\,g_a(|\mathcal{U}(t,\theta,x)|)\,\psi^{\epsilon}dtdx{} }~~~~~~~~~~~~~~~~
\nonumber\\
=
\int_{\torus^{2}}\int_{0}^{T}a\,g_a(|\mathcal{U}(t,\theta,x)|)\,\psi^{\epsilon}dtdx
+\int_{\torus^{2}}\int_{0}^{T}ab\epsilon\mathcal{M}(t,\theta,x)\,g_a(|\mathcal{U}(t,\theta,x)|)\,\psi^{\epsilon} dtdx.
\end{eqnarray}
Using the two-scale limits we have
\begin{equation}\int_{\torus^{2}}\int_{0}^{T}\int_{0}^{1}\widetilde{\mathcal{A}}\psi \,d\theta dtdx=\int_{\torus^{2}}\int_{0}^{T}\int_{0}^{1}a\,g_a(|\mathcal{U}(t,\theta,x)|)\,\psi \,dtdx.
\end{equation}
Using the expression of $\widetilde{\mathcal{C}}_{\epsilon}$ we prove either that
\begin{equation}\int_{\torus^{2}}\int_{0}^{T}\int_{0}^{1}\widetilde{\mathcal{C}}\psi\, d\theta dtdx=\int_{\torus^{2}}\int_{0}^{T}\int_{0}^{1}c\,g_c(|\mathcal{U}(t,\theta,x)|)\,\frac{\mathcal{U}(t,\theta,x)}{|\mathcal{U}(t,\theta,x)|}\psi\, dtdx,
\end{equation}
then
\begin{equation}
\widetilde{\mathcal{A}}=a\,g_a(|\mathcal{U}(t,\theta,x)|) \,\, \textrm{and}\,\,\widetilde{\mathcal{C}}=c\,g_c(|\mathcal{U}(t,\theta,x)|)\,\frac{\mathcal{U}(t,\theta,x)}{|\mathcal{U}(t,\theta,x)|}.
\end{equation}
\end{proof}
\section{A corrector result, proof of theorem \ref{th3}}
Considering (\ref{3.5}) with coefficients given by (\ref{3.1}) coupled with (\ref{3.6}) and (\ref{3.3}) coupled with (\ref{3.7}) leads to writing  \begin{equation}\label{5.1}\mathcal{A}^{\epsilon}(t,x)=\widetilde{\mathcal{A}}^{\epsilon}(t,x)+\epsilon\widetilde{\mathcal{A}}_{1}^{\epsilon}(t,x),
\end{equation}
and \begin{equation}\label{5.2}\mathcal{C}^{\epsilon}(t,x)=\widetilde{\mathcal{C}}^{\epsilon}(t,x)+\epsilon\widetilde{\mathcal{C}}_{1}^{\epsilon}(t,x),
\end{equation} 
where \begin{equation}\widetilde{\mathcal{A}}^{\epsilon}(t,x)=\widetilde{\mathcal{A}}(t,\frac{t}{\epsilon},x),\,\,
\widetilde{\mathcal{C}}^{\epsilon}(t,x)=\widetilde{\mathcal{C}}(t,\frac{t}{\epsilon},x),
\end{equation}
with $\widetilde{\mathcal{A}}$ and $\widetilde{\mathcal{C}}$ given by (\ref{EqHomcof})
and where 
\begin{equation}\widetilde{\mathcal{A}}_{1}^{\epsilon}(t,x)=\widetilde{\mathcal{A}}_{1}(t,\frac{t}{\epsilon},x),\,\,
\widetilde{\mathcal{C}}_{1}^{\epsilon}(t,x)=\widetilde{\mathcal{C}}_{1}(t,\frac{t}{\epsilon},x),
\end{equation}
with
\begin{equation} \widetilde{\mathcal{A}}_{1}(t,\theta,x)=-ab\mathcal{M}(t,\theta,x)\,g_a(|\mathcal{U}(t,\theta,x)|)\,\,
\textrm{and}\,\,\widetilde{\mathcal{C}}_{1}(t,\theta,x)=
-cb\mathcal{M}(t,\theta,x)\,g_c(|\mathcal{U}(t,\theta,x)|)\,\frac{\mathcal{U}(t,\theta,x)}{|\mathcal{U}(t,\theta,x)|}
\end{equation}
Because of assumptions (\ref{eq2}), (\ref{eq3}) and (\ref{eq4}),
\begin{gather}
\label{totitoti}
\widetilde{\mathcal{A}},\,\, \widetilde{\mathcal{C}},\,\, 
\widetilde{\mathcal{A}}_1,\,\, \widetilde{\mathcal{C}}_1,\,\,
 \widetilde{\mathcal{A}}^\epsilon, \,\,
\widetilde{\mathcal{A}}_1^\epsilon,\,\, \widetilde{\mathcal{C}}^\epsilon, \text{ and }
\widetilde{\mathcal{C}}_1^\epsilon
\text{ are regular and bounded. }
\end{gather}
Moreover, the supplementary assumption that $U_{thr}=0$ yields
\begin{equation}\label{EqRaj}\widetilde{\mathcal{A}}(t,\theta,x)\geq \widetilde G_{thr}\,\,\textrm{for any}\,\,t,\theta,x
\in [0,T]\times\R\times\torus^2.
\end{equation} 
\begin{theorem}\label{th5.1}Under assumptions (\ref{3.12}), (\ref{3.15}), (\ref{totitoti}) and (\ref{EqRaj}), considering $z^{\epsilon}$ the solution to (\ref{3.5}) with coefficients given by (\ref{3.1}) and (\ref{3.3}) coupled with (\ref{5.1}) and (\ref{5.2}) and $U^{\epsilon}=U^{\epsilon}(t,x)=U(t,\frac{t}{\epsilon},x)$ where $U$ is the solution to (\ref{eqHomIntro}), for any $T$ not depending on $\epsilon,$ the following estimate holds for $z^{\epsilon}-U^{\epsilon}$
\begin{equation}\Big\|\frac{z^{\epsilon}-U^{\epsilon}}{\epsilon}\Big\|_{  L^{\infty}([0,T),L^{2}(\torus^{2}))}\leq\alpha,
\end{equation} where $\alpha$ is a constant not depending on $\epsilon.$\\
Furthermore, sequence $\left(\frac{z^{\epsilon}-U^{\epsilon}}{\epsilon}\right)$ two-scale converges to a profile $U^{1}\in L^{\infty}([0,T],L^{\infty}_\#(\R,L^2(\torus^{2})))$ which is the unique solution to
\begin{equation}\label{eq5.5.111}
\frac{\partial U^{1}}{\partial \theta}-\nabla\cdot\left(\widetilde{\mathcal{A}}\nabla U^{1}\right)=\nabla\cdot\widetilde{\mathcal{C}}_{1}+\frac{\partial U}{\partial t}+\nabla\cdot(\widetilde{\mathcal{A}}_{1}\nabla U).
\end{equation}
\end{theorem}
\begin{proof}{\bf of theorem \ref{th3}.} Theorem \ref{th3} is a direct consequence of theorem \ref{th5.1}.
\end{proof}
\\
\begin{proof}{\bf of theorem \ref{th5.1}.}
Using (\ref{5.1}) and (\ref{5.2}) equations (\ref{3.5}) becomes
\begin{equation}\label{eq5.1}\frac{\partial z^{\epsilon}}{\partial t}-\frac{1}{\epsilon}\nabla\cdot(\widetilde{\mathcal{A}}^{\epsilon}\nabla z^{\epsilon})=\frac{1}{\epsilon}\nabla\cdot\widetilde{\mathcal{C}}^{\epsilon}+\nabla\cdot(\widetilde{\mathcal{A}}_{1}^{\epsilon}\nabla z^{\epsilon})+\nabla\cdot\widetilde{\mathcal{C}}_{1}^{\epsilon}.
\end{equation}
For $U^{\epsilon}$, we have
\begin{equation}\frac{\partial U^{\epsilon}}{\partial t}=\left(\frac{\partial U}{\partial t}\right)^{\epsilon}+\frac{1}{\epsilon}\left(\frac{\partial U}{\partial \theta}\right)^{\epsilon},
\end{equation}
where 
\begin{gather}\label{H4acopie}
   \left(\frac{\partial U}{\partial t}\right)^{\epsilon}(t,x)
   =\frac{\partial U}{\partial t}(t,\frac{t}{\epsilon},x)\,\,\textrm{ and }
     \left(\frac{\partial U}{\partial\theta}\right)^{\epsilon}(t,x)=\frac{\partial U}{\partial\theta}(t,\frac{t}{\epsilon},x).
\end{gather}
Using (\ref{ee179}), $U^{\epsilon}$ is solution to
\begin{equation}\label{eq5.2}
\frac{\partial U^{\epsilon}}{\partial t}-\frac{1}{\epsilon}\nabla\cdot\left(\widetilde{\mathcal{A}}^{\epsilon}\nabla U^{\epsilon}\right)=\frac{1}{\epsilon}\nabla\cdot\widetilde{\mathcal{C}}^{\epsilon}+\left(\frac{\partial U}{\partial t}\right)^{\epsilon}.
\end{equation}
Formulas (\ref{eq5.1}) and (\ref{eq5.2}) give
\begin{equation}\label{eq5.3}
\frac{\partial (z^{\epsilon}-U^{\epsilon})}{\partial t}-\frac{1}{\epsilon}\nabla\cdot\left(\widetilde{\mathcal{A}}^{\epsilon}\nabla (z^{\epsilon}- U^{\epsilon})\right)=\nabla\cdot\widetilde{\mathcal{C}}_{1}^{\epsilon}+\left(\frac{\partial U}{\partial t}\right)^{\epsilon}+\nabla\cdot(\widetilde{\mathcal{A}}_{1}^{\epsilon}\nabla z^{\epsilon}).
\end{equation}
Multiplying equation (\ref{eq5.3}) by $\frac{1}{\epsilon}$ and using the fact that $z^{\epsilon}=z^{\epsilon}-U^{\epsilon}+U^{\epsilon}$ in the right hand side of equation (\ref{eq5.3})$,\frac{z^{\epsilon}-U^{\epsilon}}{\epsilon}$ is solution to:
\begin{equation}\label{eq5.4}
\frac{\ds\partial \left(\frac{z^{\epsilon}-U^{\epsilon}}{\epsilon}\right)}{\partial t}-\frac{1}{\epsilon}\nabla\cdot\left((\widetilde{\mathcal{A}}^{\epsilon}+\epsilon\widetilde{\mathcal{A}}_{1}^{\epsilon})\nabla (\frac{z^{\epsilon}- U^{\epsilon}}{\epsilon})\right)=\frac{1}{\epsilon}\left(\nabla\cdot\widetilde{\mathcal{C}}_{1}^{\epsilon}+(\frac{\partial U}{\partial t})^{\epsilon}+\nabla\cdot(\widetilde{\mathcal{A}}_{1}^{\epsilon}\nabla U^{\epsilon})\right).
\end{equation}
\begin{remark} Concerning notations, we have to pay attention to the fact that 
\begin{equation}\widetilde{\mathcal{A}}^{\epsilon}\neq\widetilde{\mathcal{A}}_{\epsilon}\,\,\textrm{and}\,\,
\widetilde{\mathcal{C}}^{\epsilon}\neq\widetilde{\mathcal{C}}_{\epsilon}.\end{equation}
{\ \hfill \rule{0.5em}{0.5em}}
\end{remark}
Our aim here is to prove that $\frac{z^{\epsilon}-U^{\epsilon}}{\epsilon}$ is bounded by a constant $\alpha$ not depending on $\epsilon.$ For this let us use that
$\widetilde{\mathcal{A}}^{\epsilon},\,\,\widetilde{\mathcal{A}}_{1}^{\epsilon},\,\,
\widetilde{\mathcal{C}}^{\epsilon}\,\,\textrm{and}\,\,\widetilde{\mathcal{C}}_{1}^{\epsilon}$ are regular and bounded coefficients (see (\ref{totitoti})) and that $\widetilde{\mathcal{A}}^{\epsilon}\geq G_{thr}$ 
(see (\ref{EqRaj})). 
Thus, $\nabla\cdot\widetilde{\mathcal{C}}_{1}^{\epsilon}$ is bounded, $\nabla\cdot(\widetilde{\mathcal{A}}_{1}^{\epsilon}\nabla U^{\epsilon})$ is also bounded. Since $U^{\epsilon}$ is solution to (\ref{eq5.2}),
$\frac{\partial U}{\partial t}$ satisfies the following equation
\begin{equation}\label{eq5.5a}\frac{\ds\partial \left(\frac{\partial U}{\partial t}\right)}{\partial\theta}-\nabla\cdot\left(\widetilde{\mathcal{A}}\nabla\frac{\partial U}{\partial t}\right)=\frac{\partial\nabla\cdot\widetilde{\mathcal{C}}}{\partial t}+\nabla\cdot\left(\frac{\partial \widetilde{\mathcal{A}}}{\partial t}\nabla U\right).
\end{equation}
Equation (\ref{eq5.5a}) is linear with regular and bounded coefficients. Then using a result of Ladyzenskaja, Solonnikov and Ural'Ceva \cite{LSU}, $\frac{\partial U}{\partial t}$ is regular 
and bounded . Then the coefficients of equations (\ref{eq5.4}) are regular and bounded. 
Then, using the same arguments as in the proof of theorem \ref{th1} we obtain that  
$\left(\frac{z^{\epsilon}-U^{\epsilon}}{\epsilon}\right)$  is bounded,
that it  two-scale converges to a profile $U^{1}\in L^{\infty}([0,T],L^{\infty}_\#(\R,L^2(\torus^{2})))$ 
and that this profile $U^{1}$ satisfies equation (\ref{eq5.5}).
\end{proof}

\bibliographystyle{plain}
\bibliography{biblio}

\end{document}